\documentclass{amsart}

\usepackage{amssymb,amsxtra,latexsym}
\usepackage{mathtools}
\usepackage{mathrsfs}
\usepackage{stmaryrd}
\usepackage{upgreek}
\usepackage[numbers]{natbib}
\usepackage{url}
\usepackage[english]{babel}
\usepackage{empheq}
\usepackage{upref}
\usepackage{enumerate}

\renewcommand{\qedsymbol}{$\blacksquare$}

\newtheorem{theorem}{Theorem}[section]
\newtheorem{lemma}[theorem]{Lemma}
\newtheorem{proposition}[theorem]{Proposition}
\newtheorem{corollary}[theorem]{Corollary}

\theoremstyle{definition}

\newtheorem{example}[theorem]{Example}

\theoremstyle{remark}

\numberwithin{equation}{section}

\begin{document}

\title[On the generalized Suzuki curve]{On the Zeta function and the automorphism group of the generalized Suzuki curve}


%    author one information
\author{Herivelto Borges}
\address{Instituto de Ci\^{e}ncias Matem\'{a}ticas e de Computa\c{c}\~{a}o, Universidade de S\~{a}o Paulo, Avenida Trabalhador S\~{a}o-carlense, 400, CEP 13566--590, S\~{a}o Carlos, SP, Brazil}
\curraddr{}
\email{hborges@icmc.usp.br}

%    author two information
\author{Mariana Coutinho}
\address{Instituto de Matem\'{a}tica, Estat\'{i}stica e Computa\c{c}\~{a}o Cient\'{i}fica, Universidade Estadual de Campinas, Rua S\'{e}rgio Buarque de Holanda, 651, CEP 13083--859, Campinas, SP, Brazil}
\curraddr{}
\email{mariananery@alumni.usp.br}

\thanks{The authors would like to thank Felipe Voloch for the discussion regarding the content of Section \ref{Section: Applications, Examples, and Remarks}.}
\thanks{The first author was supported by FAPESP (Brazil), grant 2017/04681--3, and partially funded by the 2019 IMPA Post-doctoral Summer Program.}
\thanks{The second author was financed in part by Coordena\c{c}\~{a}o de Aperfei\c{c}oamento de Pessoal de N\'{i}vel Superior - Brasil (CAPES) - Finance Code 001, CNPq (Brazil), grant 154359/2016--5, and FAPESP (Brazil), grant 2018/23839--0.}

%    \subjclass is required.

%\keywords{Suzuki curve, rational points, Zeta function, L-polynomial, automorphism group}

\date{}

%    Abstract is required.
\begin{abstract}
For $p$ an odd prime number, $q_{0}=p^{t}$, and $q=p^{2t-1}$, let $\mathcal{X}_{\mathcal{G}_{\mathcal{S}}}$ be the nonsingular model of 
$$
Y^{q}-Y=X^{q_{0}}(X^{q}-X).
$$
In the present work, the number of $\mathbb{F}_{q^{n}}$-rational points and the full automorphism group of $\mathcal{X}_{\mathcal{G}_{\mathcal{S}}}$ are determined. In addition, the L-polynomial of this curve is provided, and the number of $\mathbb{F}_{q^{n}}$-rational points on the Jacobian $J_{\mathcal{X}_{\mathcal{G}_{\mathcal{S}}}}$ is used to construct \'{e}tale covers of $\mathcal{X}_{\mathcal{G}_{\mathcal{S}}}$, some with many rational points.
\end{abstract}

\maketitle


%%%%%%%%%%%%%%%%%%%%%%%%%%%%%%%%%%%%%%%%%%%%%%%%%%%%%%%%%%
\section{Introduction}\label{Section: Introduction}

\noindent Algebraic curves over finite fields is a significant research topic, in particular because of its connection with other areas of mathematics. In this context, determining the number of rational points on a curve is a classical, but often challenging, problem. While a general method to compute such numbers is out of reach, effective bounds can be found in the literature. For instance, if $\mathcal{Y}$ is a (projective, nonsingular, geometrically irreducible, algebraic) curve of genus $g$ defined over the finite field $\mathbb{F}_{q}$, then the remarkable Hasse-Weil bound gives
\begin{equation*}
|\text{N}_{q}(\mathcal{Y})-(q+1)|\leqslant 2gq^{1/2},
\end{equation*}
where $q$ is a power of a prime $p$, and $\text{N}_{q}(\mathcal{Y})$ is the number of $\mathbb{F}_{q}$-rational points on $\mathcal{Y}$.

Curves attaining the previous upper (resp. lower) bound are called $\mathbb{F}_{q}$-maximal (resp. $\mathbb{F}_{q}$-minimal). For $p=2$, an important example is the Deligne-Lusztig curve associated with the Suzuki group $\text{Sz}(q)$ \cite{DELIGNE-LUSZTIG}, \cite{HANSEN}, \cite{HANSEN-STICHTENOTH}, here for simplicity called the Suzuki curve, which is the nonsingular model $\mathcal{Y}_{\mathcal{S}}$ of
\begin{equation*}
\mathcal{S}:Y^{q}-Y=X^{q_{0}}(X^{q}-X),
\end{equation*}
where $q_{0}=2^{t}$, $q=2^{2t-1}$, and $t\geqslant 2$\footnote{This is an alternative plane model for the Suzuki curve, which is usually given by $Y^{q}-Y=X^{q_{0}}(X^{q}-X),$ with $q_{0}=p^{t}$ and $q=p^{2t+1}$ (see \cite[Proposition $4.2$]{GS}).}. Indeed in \cite[Proposition $4.3$]{HANSEN}, together with the expression for the Zeta function of $\mathcal{Y}_{\mathcal{S}}$, the explicit formula for the number of rational points on $\mathcal{Y}_{\mathcal{S}}$ shows that it is $\mathbb{F}_{q^{4}}$-maximal.

The Suzuki curve is optimal in the sense that its number of $\mathbb{F}_{q}$-rational points coincides with the maximum number of $\mathbb{F}_{q}$-rational points that a curve of its genus can have \cite[Proposition $2.1$]{HANSEN-STICHTENOTH}. Moreover, in \cite[Theorem $5.1$]{FUHRMANN-TORRES}, it is shown that this curve can be characterized by its genus and number of $\mathbb{F}_{q}$-rational points. 

In addition to its maximality and optimality, the Suzuki curve is known for its large automorphism group. Specifically, it is one of four examples of curves of genus $g\geqslant 2$ which have an automorphism group of order greater than or equal to $8g^{3}$ \cite{HENN}, \cite[Theorem $11.127$]{HKT}. Furthermore, in \cite{BASSAetal} and \cite{GKT}, the investigation of genera and plane models for quotients of the Suzuki curve is addressed. Past studies also examined embeddings in $\mathbb{P}^{N}$ \cite{BR}, \cite{DE}, class field theory \cite{LAUTER}, the invariant $a$-number \cite{RACHELPRIESa-numbers}, and Weierstrass semigroups and coding theory \cite{BMZ}, \cite{CD}, \cite{DP}, \cite{EIDetal}, \cite{GK}, \cite{KP}, \cite{MATTHEWS}. More recently, the construction of a Galois cover of the Suzuki curve in \cite{SKABELUND} raised a number of other issues to be investigated \cite{GIULIETTIetal}, \cite{MTZ}.

Most of the aforementioned applications were based on knowledge of the number of $\mathbb{F}_{q^{n}}$-rational points and the automorphism group of the Suzuki curve. Motivated by this, for $p>2$, $q_{0}=p^{t}$, and $q=p^{m}$, where $m$, $t$ are positive integers satisfying $m=2t-1$, let $\mathcal{G}_{\mathcal{S}}$ be the projective geometrically irreducible plane curve (see Proposition \ref{Proposition: properties of the plane generalized Suzuki curve}) defined over $\mathbb{F}_{q}$ given in affine coordinates by
\begin{equation}
\label{Equation: generalized Suzuki curve}
\mathcal{G}_{\mathcal{S}}:Y^{q}-Y=X^{q_{0}}(X^{q}-X),
\end{equation}
and let $\mathcal{X}_{\mathcal{G}_{\mathcal{S}}}$ be its nonsingular model, called herein the generalized Suzuki curve. 

The objective of this work is twofold. First, the number of $\mathbb{F}_{q^{n}}$-rational points on $\mathcal{X}_{\mathcal{G}_{\mathcal{S}}}$ is investigated for all $n\geqslant 1$\footnote{For $n=1$, the number $\text{N}_{q^{n}}(\mathcal{X}_{\mathcal{G}_{\mathcal{S}}})$ was previously studied in \cite{PS} in connection with geometric Goppa codes. Also, for $p=3$, and considering \cite[Proposition $4.2$]{GS}, the curve $\mathcal{X}_{\mathcal{G}_{\mathcal{S}}}$ is $\mathbb{F}_{q}$-covered by the so-called Ree curve, and then the information on its maximality given by Theorem \ref{Theorem: the number of Fqn rational points on XGS} could be recovered by Serre's result (see Theorem \ref{Theorem: theorem of Serre}).}. In addition, the L-polynomial of $\mathcal{X}_{\mathcal{G}_{\mathcal{S}}}$ is determined. An explicit description of the L-polynomial of a curve $\mathcal{Y}$ is highly desirable since it encodes significant information as the number of rational points on the Jacobian $J_{\mathcal{Y}}$ of $\mathcal{Y}$. Using this fact, some constructions of curves with many rational points are presented in \cite{V}. Moreover, considering the Frobenius endomorphism $\Phi$ on $J_{\mathcal{Y}}$, then the characteristic polynomial of $\Phi$ is described precisely by the reciprocal of the L-polynomial of $\mathcal{Y}$, and from its factorization, the Frobenius linear series on $\mathcal{Y}$ is obtained. For further details, see \cite[Sections $9.7$ and $9.8$]{HKT}.

Second, the full automorphism group of $\mathcal{X}_{\mathcal{G}_{\mathcal{S}}}$ is presented. Additionally, considering the more general curve $\mathcal{X}$ given by the nonsingular model of
\begin{equation*}
Y^{q}-Y=X^{q_{0}}(X^{q}-X),
\end{equation*}
where  $q=p^{m}$ and $q_{0}=p^{t}$, with $m$, $t$ being positive integers satisfying $2t>m\geqslant t$, and $2t-1>m\geqslant t$ if $p=2$, one can easily verify that the proof of Theorem \ref{Theorem: the automorphism group of XGS} below also holds for $\mathcal{X}$. Consequently, the curves $\mathcal{X}$ and $\mathcal{X}_{\mathcal{G}_{\mathcal{S}}}$ have an isomorphic automorphism group when $p>2$. Further, for $p=2$, this extension of Theorem \ref{Theorem: the automorphism group of XGS} answers a question raised by Giulietti and Korchm\'aros, completing their description of the full automorphism group of $\mathcal{X}$ in \cite{GK1}.

Let $p$ be an odd prime. For each integer $k$, define
\begin{equation}
\label{Equation: definition of quadratic character}
\upeta(k)\coloneqq\bigg(\dfrac{k}{p}\bigg)=\left\{\begin{array}{rl}
1,& \text{ if } p\nmid k\text{ and } k \text{ is a square modulo } p\\
0,& \text{ if } p\mid k\\
-1,& \text{ otherwise},
\end{array}\right.
\end{equation}
where $\bigg(\dfrac{\ast}{p}\bigg)$ is the Legendre symbol, and consider
\begin{equation}
\label{Equation: definition of p*1/2}
\tilde{p}^{1/2}\coloneqq\left\{\begin{array}{rl}
p^{1/2},& \text{ if }\upeta(-1)=1\\
\textbf{i}p^{1/2},& \text{ if }\upeta(-1)=-1.
\end{array}\right.
\end{equation}

The main results are the following.

%-------------------------------------------------------------------------------------------------------------------------------------------------------------
\begin{theorem}
\label{Theorem: the number of Fqn rational points on XGS}
If $g$ denotes the genus of $\mathcal{X}_{\mathcal{G}_{\mathcal{S}}}$, then the number $\emph{N}_{q^{n}}(\mathcal{X}_{\mathcal{G}_{\mathcal{S}}})$ of $\mathbb{F}_{q^{n}}$-rational points on $\mathcal{X}_{\mathcal{G}_{\mathcal{S}}}$ is described as follows.
\begin{enumerate}
%------------------------------------------------------------------------------
\item[\emph{1.}] If $p\mid n$, then
\begin{equation*}
\emph{N}_{q^{n}}(\mathcal{X}_{\mathcal{G}_{\mathcal{S}}})=\left\{\begin{array}{cl}
q^{n}+1,& \text{ if }n\text{ is odd}\\
q^{n}+1-\upeta((-1)^{n/2})\cdot 2gq^{n/2},& \text{ if }n\text{ is even}.
\end{array}\right.
\end{equation*}
%------------------------------------------------------------------------------
\item[\emph{2.}] If $p\nmid n$, then 
\begin{equation*}
\emph{N}_{q^{n}}(\mathcal{X}_{\mathcal{G}_{\mathcal{S}}})=\left\{\begin{array}{cl}
q^{n}+1+\upeta((-1)^{(n-1)/2}n)\cdot 2gq^{n/2}p^{-1/2},& \text{ if }n\text{ is odd}\\
q^{n}+1,& \text{ if }n\text{ is even}.
\end{array}\right.
\end{equation*}
\end{enumerate}
In particular, $\mathcal{X}_{\mathcal{G}_{\mathcal{S}}}$ is $\mathbb{F}_{q^{n}}$-maximal if and only if $p\mid n$, $n\equiv 2\pmod{4}$, and $p\equiv 3\pmod{4}$.
\end{theorem}
%-------------------------------------------------------------------------------------------------------------------------------------------------------------

%-------------------------------------------------------------------------------------------------------------------------------------------------------------
\begin{theorem}
\label{Theorem: the L-polynomial of XGS}
Let $g$ be the genus of $\mathcal{X}_{\mathcal{G}_{\mathcal{S}}}$, and let $\mathcal{M}_{p}(T)$ be the minimal polynomial of
$
-\zeta_{p}/\tilde{p}^{1/2}
$
over $\mathbb{Q}$, where $\zeta_{p}$ is the primitive $p$-th root of unity $\mathbf{e}^{\frac{2\pi\mathbf{i}}{p}}$. Then, the L-polynomial $L_{\mathcal{X}_{\mathcal{G}_{\mathcal{S}}}/\mathbb{F}_{q}}(T)$ of $\mathcal{X}_{\mathcal{G}_{\mathcal{S}}}$ is given by
\begin{equation}
\label{Equation: the L-polynomial over Fq}
L_{\mathcal{X}_{\mathcal{G}_{\mathcal{S}}}/\mathbb{F}_{q}}(T)=\bigg(p^{p-1}\cdot(qT^{2}-\upeta(-1))\cdot \mathcal{M}_{p}(p^{t-1}T)^{2}\bigg)^{\frac{g}{p}}.
\end{equation}
Moreover: 
\begin{enumerate}
%------------------------------------------------------------------------------
\item[\emph{1.}] If $p\mid n$, then $L_{\mathcal{X}_{\mathcal{G}_{\mathcal{S}}}/\mathbb{F}_{q^{n}}}(T)$ is given by
\begin{enumerate}
\item[\emph{a)}] $\bigg(q^{n}T^{2}-\upeta(-1)\bigg)^{g}$, if $n$ is odd.
\item[\emph{b)}] $\bigg(q^{n/2}T-\upeta((-1)^{n/2})\bigg)^{2g}$, if $n$ is even.
\end{enumerate}
%------------------------------------------------------------------------------
\item[\emph{2.}] If $p\nmid n$, then $L_{\mathcal{X}_{\mathcal{G}_{\mathcal{S}}}/\mathbb{F}_{q^{n}}}(T)$ is given by
\begin{enumerate}
\item[\emph{a)}] $\bigg(p^{p-1}\cdot(q^{n}T^{2}-\upeta(-1))\cdot \mathcal{M}_{p}(\upeta((-1)^{(n-1)/2}n)\cdot p^{t-1}q^{(n-1)/2}T)^{2}\bigg)^{\frac{g}{p}}$, if $n$ is odd.
\item[\emph{b)}] $\Bigg(q^{\frac{np}{2}}T^{p}-\upeta((-1)^{n/2})\Bigg)^{\frac{2g}{p}}$, if $n$ is even.
\end{enumerate}
\end{enumerate}
\end{theorem}
%-------------------------------------------------------------------------------------------------------------------------------------------------------------

Finally, if $\overline{\mathbb{F}_{q}}(\mathcal{X}_{\mathcal{G}_{\mathcal{S}}})$ is the function field of $\mathcal{X}_{\mathcal{G}_{\mathcal{S}}}$, and $x,y\in\overline{\mathbb{F}_{q}}(\mathcal{X}_{\mathcal{G}_{\mathcal{S}}})$ are such that $\overline{\mathbb{F}_{q}}(\mathcal{X}_{\mathcal{G}_{\mathcal{S}}})=\overline{\mathbb{F}_{q}}(x,y)$ and $y^{q}-y=x^{q_{0}}(x^{q}-x)$, then the automorphism group of $\mathcal{X}_{\mathcal{G}_{\mathcal{S}}}$ is described as follows.

%-------------------------------------------------------------------------------------------------------------------------------------------------------------
\begin{theorem}
\label{Theorem: the automorphism group of XGS}
The automorphism group $\emph{Aut}(\mathcal{X}_{\mathcal{G}_{\mathcal{S}}})$ of $\mathcal{X}_{\mathcal{G}_{\mathcal{S}}}$ has order $q^{2}(q-1)$ and is given by the maps
\begin{equation*}
(x,y)\mapsto (\alpha x+\beta,\alpha \beta^{q_{0}}x+\alpha^{q_{0}+1}y+\gamma),
\end{equation*}
where $\alpha\in\mathbb{F}_{q}^{\ast}$ and $\beta,\gamma\in\mathbb{F}_{q}$. Moreover, 
\begin{equation*}
\emph{Aut}(\mathcal{X}_{\mathcal{G}_{\mathcal{S}}})=\mathbb{G}\rtimes\mathbb{H},
\end{equation*}
where $\mathbb{G}$ is the Sylow $p$-subgroup of $\emph{Aut}(\mathcal{X}_{\mathcal{G}_{\mathcal{S}}})$ consisting of the maps
\begin{equation*}
(x,y)\mapsto (x+\beta,\beta^{q_{0}}x+y+\gamma),
\end{equation*}
with $\beta,\gamma\in\mathbb{F}_{q}$, and $\mathbb{H}$ is the cyclic complement of $\mathbb{G}$ in $\emph{Aut}(\mathcal{X}_{\mathcal{G}_{\mathcal{S}}})$, which can be described by the maps
\begin{equation*}
(x,y)\mapsto (\alpha x,\alpha^{q_{0}+1}y),
\end{equation*}
with $\alpha\in\mathbb{F}_{q}^{\ast}$.
\end{theorem}
%-------------------------------------------------------------------------------------------------------------------------------------------------------------

This paper is organized as follows. Section \ref{Section: Preliminaries} provides the general background used to prove Theorems \ref{Theorem: the number of Fqn rational points on XGS}, \ref{Theorem: the L-polynomial of XGS}, and \ref{Theorem: the automorphism group of XGS}, which is done in Sections \ref{Section: Proof of Theorem 1}, \ref{Section: Proof of Theorem 2}, and \ref{Section: Proof of Theorem 3}, respectively. Finally, in Section \ref{Section: Applications, Examples, and Remarks} some examples and applications of Theorems \ref{Theorem: the number of Fqn rational points on XGS} and \ref{Theorem: the L-polynomial of XGS} are considered.

%%%%%%%%%%%%%%%%%%%%%%%%%%%%%%%%%%%%%%%%%%%%%%%%%%%%%%%%%%
\begin{center}
\bf{Notation}
\end{center}

\noindent Together with \eqref{Equation: definition of quadratic character} and \eqref{Equation: definition of p*1/2}, the following notation is used throughout this text.

\begin{itemize}
\item $p$ is a prime number, $q=p^{m}$, and $q_{0}=p^{t}$ for some positive integers $m,t$.
\item For each positive integer $n$, $\mathbb{F}_{q^{n}}$ is the finite field with $q^{n}$ elements, $\mathbb{F}_{q^{n}}^{\ast}\coloneqq\mathbb{F}_{q^{n}}\setminus\{0\}$, $\overline{\mathbb{F}_{q}}$ is the algebraic closure of $\mathbb{F}_{q^{n}}$, and $\overline{\mathbb{F}_{q}}^{\ast}\coloneqq\overline{\mathbb{F}_{q}}\setminus\{0\}$.
\item $\text{Tr}_{\mathbb{F}_{q^{n}}/\mathbb{F}_{p}}$ and $\text{N}_{\mathbb{F}_{q^{n}}/\mathbb{F}_{p}}$ denote the absolute trace and norm functions of $\mathbb{F}_{q^{n}}$, respectively.
\item The term \textit{curve} (resp. \textit{plane curve}) means a projective, nonsingular, geometrically irreducible, algebraic curve (resp. a projective plane algebraic curve).
\item For a curve $\mathcal{Y}$ defined over $\mathbb{F}_{q}$
\begin{itemize}
\item $\overline{\mathbb{F}_{q}}(\mathcal{Y})$ is the function field of $\mathcal{Y}$, and $\mathbb{F}_{q^{n}}(\mathcal{Y})$ is the $\mathbb{F}_{q^{n}}$-rational function field of $\mathcal{Y}$;
\item $\mathcal{Y}(\mathbb{F}_{q^{n}})$ is the set of $\mathbb{F}_{q^{n}}$-rational points on $\mathcal{Y}$, and $\text{N}_{q^{n}}(\mathcal{Y})=\#\mathcal{Y}(\mathbb{F}_{q^{n}})$;
\item $L_{\mathcal{Y}/\mathbb{F}_{q^{n}}}(T)$ is the L-polynomial of $\mathcal{Y}$ viewed as a curve defined over $\mathbb{F}_{q^{n}}$;
\item $J_{\mathcal{Y}}$ is the Jacobian of $\mathcal{Y}$, and $J_{\mathcal{Y}}(\mathbb{F}_{q^{n}})$ is the group of $\mathbb{F}_{q^{n}}$-rational points on $J_{\mathcal{Y}}$. Also, assuming that $\mathcal{Y}$ has a rational point $P_{0}\in\mathcal{Y}(\mathbb{F}_{q})$, and considering $\mathcal{Y}$ embedded in $J_{\mathcal{Y}}$ via
$$
P_{0}\mapsto 0,
$$
$\langle \mathcal{Y}(\mathbb{F}_{q^{n}}) \rangle$ is the subgroup of $J_{\mathcal{Y}}(\mathbb{F}_{q^{n}})$ generated by $\mathcal{Y}(\mathbb{F}_{q^{n}})$.
\end{itemize}
\item $\mathbf{e}$ is the Euler's number, $\mathbf{i}$ is the imaginary unity, and $\zeta_{k}\coloneqq\mathbf{e}^{\frac{2\pi\mathbf{i}}{k}}$ for each positive integer $k$.
\item For each positive integer $k$, 
\begin{equation*}
\mu_{k}\coloneqq\bigg\{\zeta_{k}^{i}\,\,:\,\,1\leqslant i\leqslant k\text{ and }\gcd(i,k)=1\bigg\}
\end{equation*}
is the set of roots of the $k$-th cyclotomic polynomial $\Phi_{k}(T)$.
\end{itemize}

%%%%%%%%%%%%%%%%%%%%%%%%%%%%%%%%%%%%%%%%%%%%%%%%%%%%%%%%%%
\section{Preliminaries}\label{Section: Preliminaries}

%/////////////////////////////////////////////////////////////////////////////////////////////////////////////////////////////////////////////////////////////
\subsection{L-Polynomials and Supersingular Curves}\label{Subsection: L-polynomials supersingular curves}

Let $\mathcal{Y}$ be a curve of genus $g$ defined over the finite field $\mathbb{F}_{q}$. It is well known that the L-polynomial $L_{\mathcal{Y}/\mathbb{F}_{q}}(T)\in\mathbb{Z}[T]$ has degree $2g$ and satisfies $L_{\mathcal{Y}/\mathbb{F}_{q}}(0)=1$ \cite[Chapter $9$]{HKT}. Further, if $\omega_{1},\ldots,\omega_{2g}\in\mathbb{C}$ are the roots of the reciprocal polynomial of $L_{\mathcal{Y}/\mathbb{F}_{q}}(T)$, then the following holds:
\begin{equation}
\label{Equation: number of rational points and Weil-number - first form}
\text{N}_{q^{n}}(\mathcal{Y})=q^{n}+1-\sum_{i=1}^{2g}\omega_{i}^{n}.
\end{equation}

From the Hasse-Weil theorem, $|\omega_{i}|=q^{1/2}$ for all $i=1,\ldots,2g$, and then $\omega_{i}=q^{1/2}\xi_{i}$, with $|\xi_{i}|=1$ for each $i$. In particular, \eqref{Equation: number of rational points and Weil-number - first form} can be rewritten as
\begin{equation*}
\text{N}_{q^{n}}(\mathcal{Y})=q^{n}+1-q^{n/2}\sum_{i=1}^{2g}\xi_{i}^{n},
\end{equation*}
and thus $\mathcal{Y}$ is $\mathbb{F}_{q^{n}}$-maximal (resp. $\mathbb{F}_{q^{n}}$-minimal) if and only if $\xi_{i}^{n}=-1$ (resp. $\xi_{i}^{n}=1$) for each $i\in\{1,\ldots,2g\}$.

The curve $\mathcal{Y}$ is supersingular if $\xi_{i}$ is a root of unity for all $i=1,\ldots,2g$. In this case, the number
\begin{equation*}
s\coloneqq\min\bigg\{n\,\,:\,\,\xi_{i}^{n}=1 \text{ for all }i=1,\ldots,2g\bigg\}
\end{equation*}
is called the period of $\mathcal{Y}$ over $\mathbb{F}_{q}$.

The following result is used in the proof of Theorem \ref{Theorem: the number of Fqn rational points on XGS}. 

%-------------------------------------------------------------------------------------------------------------------------------------------------------------
\begin{theorem}[\citealp{GMc}, Theorem $1$]
\label{Theorem: Gary McGuire and Yilmaz' theorem}
For $q$ odd, let $\mathcal{Y}$ be a curve defined over $\mathbb{F}_{q}$. If $\mathcal{Y}$ is supersingular with period $s$ over $\mathbb{F}_{q}$, and $n=n_{1}k$, where $n_{1}=\gcd(n, s)$, then the following occurs:
\begin{enumerate}
%------------------------------------------------------------------------------
\item[$1.$] If $n_{1}m$ is even, then
\begin{equation*}
\emph{N}_{q^{n}}(\mathcal{Y})-(q^{n}+1)=q^{(n-n_{1})/2}\bigg[\emph{N}_{q^{n_{1}}}(\mathcal{Y})-(q^{n_{1}}+1)\bigg].
\end{equation*}
%------------------------------------------------------------------------------
\item[$2.$] If $n_{1}m$ is odd and $p\mid k$, then
\begin{equation*}
\emph{N}_{q^{n}}(\mathcal{Y})-(q^{n}+1)=q^{(n-n_{1})/2}\bigg[\emph{N}_{q^{n_{1}}}(\mathcal{Y})-(q^{n_{1}}+1)\bigg].
\end{equation*}
%------------------------------------------------------------------------------
\item[$3.$]  If $n_{1}m$ is odd and $p\nmid k$, then
\begin{equation*}
\emph{N}_{q^{n}}(\mathcal{Y})-(q^{n}+1)=\upeta((-1)^{(k-1)/2}k)\cdot q^{(n-n_{1})/2}\bigg[\emph{N}_{q^{n_{1}}}(\mathcal{Y})-(q^{n_{1}}+1)\bigg].
\end{equation*}
\end{enumerate}
\end{theorem}
%-------------------------------------------------------------------------------------------------------------------------------------------------------------

Finally, if $\mathcal{Z}$ is a curve defined over $\mathbb{F}_{q}$ which is $\mathbb{F}_{q}$-covered by $\mathcal{Y}$, then the following theorem \cite[Proposition $5$]{AUBRY-PERRET} provides a relation between $L_{\mathcal{Y}/\mathbb{F}_{q}}(T)$ and $L_{\mathcal{Z}/\mathbb{F}_{q}}(T)$.

%-------------------------------------------------------------------------------------------------------------------------------------------------------------
\begin{theorem}[Serre]
\label{Theorem: theorem of Serre}
Let $\mathcal{Z}$ be a curve defined over $\mathbb{F}_{q}$, and let $\varphi:\mathcal{Y}\rightarrow\mathcal{Z}$ be an $\mathbb{F}_{q}$-rational covering. Then, $L_{\mathcal{Z}/\mathbb{F}_{q}}(T)$ divides $L_{\mathcal{Y}/\mathbb{F}_{q}}(T)$. Further, suppose that $\mathcal{Y}$ is either $\mathbb{F}_{q}$-maximal or $\mathbb{F}_{q}$-minimal. Then, $\mathcal{Y}$ is $\mathbb{F}_{q}$-maximal \emph{(}resp. $\mathbb{F}_{q}$-minimal\emph{)} if and only if $\mathcal{Z}$ is $\mathbb{F}_{q}$-maximal \emph{(}resp. $\mathbb{F}_{q}$-minimal\emph{)}.
\end{theorem}
%-------------------------------------------------------------------------------------------------------------------------------------------------------------

%/////////////////////////////////////////////////////////////////////////////////////////////////////////////////////////////////////////////////////////////
\subsection{On the Number of $\mathbb{F}_{q^{n}}$-Rational Points of $Y^{p}-Y=XR(X)$}\label{Subsection: On the number of Fqn-rational points of Yp-Y=XR(X)}

Suppose that $p$ is an odd prime number. For each positive integer $n$, the number of $\mathbb{F}_{q^{n}}$-rational points on a nonsingular model $\mathcal{Y}_{R}$ of the plane curve given in affine coordinates by
\begin{equation*}
\mathcal{F}_{R}:Y^{p}-Y=XR(X),
\end{equation*}
where $R(X)$ is a $p$-polynomial defined over $\mathbb{F}_{q}$; that is,
\begin{equation*}
R(X)=\alpha_{k}X^{p^{k}}+\alpha_{k-1}X^{p^{k-1}}+\cdots+\alpha_{1}X^{p}+\alpha_{0}X,
\end{equation*}
with $\alpha_{0},\ldots,\alpha_{k}\in\mathbb{F}_{q}$ and $\alpha_{k}\neq 0$, is studied. 

The following proposition presents some properties of $\mathcal{F}_{R}$.

%-------------------------------------------------------------------------------------------------------------------------------------------------------------
\begin{proposition}[\citealp{HKT}, Lemma $12.1$]
\label{Proposition: properties of the plane curve FR}
The plane curve $\mathcal{F}_{R}$ is geometrically irreducible. Moreover:
\begin{enumerate}
%------------------------------------------------------------------------------
\item[$1.$] The genus $g$ of $\mathcal{F}_{R}$ is given by
\begin{equation*}
g=\frac{p^{k}(p-1)}{2}.
\end{equation*}
%------------------------------------------------------------------------------
\item[$2.$] $P=(0:1:0)\in\mathbb{P}^{2}(\overline{\mathbb{F}_{q}})$ is the unique point at infinity of $\mathcal{F}_{R}$. If $k=1$, then $\mathcal{F}_{R}$ is nonsingular. Otherwise, $P$ is the unique singular point of $\mathcal{F}_{R}$. In both cases, $P$ has multiplicity $p^{k}+1-p$ and is the center of only one $\mathbb{F}_{q}$-rational branch of $\mathcal{F}_{R}$.
\end{enumerate}
\end{proposition}
%-------------------------------------------------------------------------------------------------------------------------------------------------------------

Therefore, the number of $\mathbb{F}_{q^{n}}$-rational points on $\mathcal{Y}_{R}$ is given by the number of solutions in $\mathbb{F}_{q^{n}}^{2}$ of the equation
\begin{equation}
\label{Equation: number of solutions in Fqn of Yp-Y=XR(X)}
Y^{p}-Y=XR(X)
\end{equation}
plus $1$. 

To count the number of solutions in $\mathbb{F}_{q^{n}}^{2}$ of \eqref{Equation: number of solutions in Fqn of Yp-Y=XR(X)}, set 
\begin{equation*}
\begin{array}{cccc}
B_{R}^{(n)}:&\mathbb{F}_{q^{n}}\times\mathbb{F}_{q^{n}}&\rightarrow&\mathbb{F}_{p}\\
&(\alpha,\beta)&\mapsto&\dfrac{1}{2}\text{Tr}_{\mathbb{F}_{q^{n}}/\mathbb{F}_{p}}(\alpha R(\beta)+\beta R(\alpha))
\end{array}.
\end{equation*}

One can check that $B_{R}^{(n)}$ is an $\mathbb{F}_{p}$-symmetric bilinear form on $\mathbb{F}_{q^{n}}$, with associated quadratic form
\begin{equation*}
\begin{array}{cccc}
Q_{R}^{(n)}:&\mathbb{F}_{q^{n}}&\rightarrow&\mathbb{F}_{p}\\
&\alpha&\mapsto&\text{Tr}_{\mathbb{F}_{q^{n}}/\mathbb{F}_{p}}(\alpha R(\alpha))
\end{array},
\end{equation*}
which satisfies $Q_{R}^{(n)}(\gamma \alpha) = \gamma^{2}Q_{R}^{(n)}(\alpha)$ for all $\gamma\in\mathbb{F}_{p}$, and also
\begin{equation*}
B_{R}^{(n)}(\alpha,\beta)=\frac{1}{2}Q_{R}^{(n)}(\alpha+\beta)-\frac{1}{2}Q_{R}^{(n)}(\alpha)-\frac{1}{2}Q_{R}^{(n)}(\beta)
\end{equation*}
for all $(\alpha,\beta)\in\mathbb{F}_{q^{n}}^{2}$. 

Considering the radical (kernel) of $B_{R}^{(n)}$
\begin{equation*}
W_{R}^{(n)}=\bigg\{\alpha\in\mathbb{F}_{q^{n}}\,\,:\,\, B_{R}^{(n)}(\alpha,\beta)=0 \text{ for each }\beta\in\mathbb{F}_{q^{n}}\bigg\},
\end{equation*}
which is an $\mathbb{F}_{p}$-linear subspace of $\mathbb{F}_{q^{n}}$, the following expressions give the number of solutions in $\mathbb{F}_{q^{n}}^{2}$ of \eqref{Equation: number of solutions in Fqn of Yp-Y=XR(X)}.

%-------------------------------------------------------------------------------------------------------------------------------------------------------------
\begin{proposition}[\citealp{BOUWetal}, Proposition $2.6$]
\label{Proposition: number of solutions as a function of the dimension}
The number of solutions in $\mathbb{F}_{q^n}^{2}$ of $\eqref{Equation: number of solutions in Fqn of Yp-Y=XR(X)}$ is
 \begin{enumerate}
 %-----------------------------------------------------------------------------
 \item[$1.$] $q^{n}$, if $mn-\dim_{\mathbb{F}_{p}}(W_{R}^{(n)})$ is odd.
 %-----------------------------------------------------------------------------
 \item[$2.$] $q^{n}\pm(p-1)p^{\dim_{\mathbb{F}_{p}}(W_{R}^{(n)})/2}q^{n/2}$, if $mn-\dim_{\mathbb{F}_{p}}(W_{R}^{(n)})$ is even.
 \end{enumerate}
\end{proposition}
%-------------------------------------------------------------------------------------------------------------------------------------------------------------

The following result provides an important tool to determine the dimension of $W_{R}^{(n)}$.

%-------------------------------------------------------------------------------------------------------------------------------------------------------------
\begin{proposition}[\citealp{BOUWetal}, Proposition $2.1$]
\label{Proposition: characterization of WRn}
$W_{R}^{(n)}$ consists of the roots in $\mathbb{F}_{q^{n}}$ of the polynomial $$E_R(T)=R(T)^{p^k}+\sum_{i=0}^{k}(\alpha_{i}T)^{p^{k-i}}.$$ 
\end{proposition}
%-------------------------------------------------------------------------------------------------------------------------------------------------------------

%/////////////////////////////////////////////////////////////////////////////////////////////////////////////////////////////////////////////////////////////
\subsection{The Curve $\mathcal{G}_{\mathcal{S}}$}\label{Subsection: The curve GS}

Let $\mathcal{G}_{\mathcal{S}}$ be given as in \eqref{Equation: generalized Suzuki curve}. From \cite[Theorem $1$]{PS}, \cite[Proposition $4.1$]{GS}, and a straightforward calculation, the following holds.

%-------------------------------------------------------------------------------------------------------------------------------------------------------------
\begin{proposition}
\label{Proposition: properties of the plane generalized Suzuki curve}
$\mathcal{G}_{\mathcal{S}}$ is a geometrically irreducible plane curve of genus $g$ given by
\begin{equation*}
g=\frac{q_{0}(q-1)}{2}.
\end{equation*}
Further, $P=(0:1:0)\in\mathbb{P}^{2}(\overline{\mathbb{F}_{q}})$ is the unique point at infinity of $\mathcal{G}_{\mathcal{S}}$, which is also its only singular point, having multiplicity equal to $q_{0}$.
\end{proposition}
%-------------------------------------------------------------------------------------------------------------------------------------------------------------

Considering the nonsingular model $\mathcal{X}_{\mathcal{G}_{\mathcal{S}}}$ of $\mathcal{G}_{\mathcal{S}}$, let $x,y\in\overline{\mathbb{F}_{q}}(\mathcal{X}_{\mathcal{G}_{\mathcal{S}}})$ be such that $\overline{\mathbb{F}_{q}}(\mathcal{X}_{\mathcal{G}_{\mathcal{S}}})=\overline{\mathbb{F}_{q}}(x,y)$ and $y^{q}-y=x^{q_{0}}(x^{q}-x)$. Then, the following occurs.

%-------------------------------------------------------------------------------------------------------------------------------------------------------------
\begin{proposition}[\citealp{DEOLALIKAR}, Theorem $3.3$]
\label{Proposition: totally ramified place}
The extension $\overline{\mathbb{F}_{q}}(x,y)/\overline{\mathbb{F}_{q}}(x)$ is a Galois extension of degree $q$. Moreover, the only ramified place of $\overline{\mathbb{F}_{q}}(x)$ is the infinite place $\mathcal{P}_{\infty}$, which is totally ramified in $\overline{\mathbb{F}_{q}}(x,y)$.
\end{proposition}
%-------------------------------------------------------------------------------------------------------------------------------------------------------------

Let $Q_{\infty}$ be the $\mathbb{F}_{q}$-rational point of $\mathcal{X}_{\mathcal{G}_{\mathcal{S}}}$ corresponding to the unique place $\mathcal{Q}_{\infty}$ of $\overline{\mathbb{F}_{q}}(x,y)$ lying over the infinity place $\mathcal{P}_{\infty}$ of $\overline{\mathbb{F}_{q}}(x)$. Then, the following holds.

%-------------------------------------------------------------------------------------------------------------------------------------------------------------
\begin{corollary}
\label{Corollary: P is the center of a unique branch}
$P=(0:1:0)$ is the center of a unique $\mathbb{F}_{q}$-rational branch of $\mathcal{G}_{\mathcal{S}}$, namely the branch of $\mathcal{G}_{\mathcal{S}}$ corresponding to the point $Q_{\infty}\in\mathcal{X}_{\mathcal{G}_{\mathcal{S}}}$.
\end{corollary}
%-------------------------------------------------------------------------------------------------------------------------------------------------------------

This subsection ends with the following result.

%-------------------------------------------------------------------------------------------------------------------------------------------------------------
\begin{proposition}[\citealp{PS}, Theorem $7$]
\label{Proposition: some Riemann-Roch spaces}
The sets $\{1,x\}$ and $\{1,x,y\}$ are bases for the Riemann-Roch spaces $\mathscr{L}(qQ_{\infty})$ and $\mathscr{L}((q+q_0)Q_{\infty})$, respectively.
\end{proposition}
%-------------------------------------------------------------------------------------------------------------------------------------------------------------

%/////////////////////////////////////////////////////////////////////////////////////////////////////////////////////////////////////////////////////////////
\subsection{Automorphism Group}\label{Subsection: Automorphism group}

Let $\mathcal{Y}$ be a curve of genus $g$ defined over the finite field $\mathbb{F}_{q}$. The automorphism group $\text{Aut}(\mathcal{Y})$ of $\mathcal{Y}$ is defined as the group of $\overline{\mathbb{F}_{q}}$-automorphisms of the function field $\overline{\mathbb{F}_{q}}(\mathcal{Y})$. 

Considering the action of $\text{Aut}(\mathcal{Y})$ on the points of $\mathcal{Y}$, the stabilizer of $\text{Aut}(\mathcal{Y})$ at $Q\in\mathcal{Y}$, denoted by $\text{Aut}_{Q}(\mathcal{Y})$, is the subgroup of $\text{Aut}(\mathcal{Y})$ consisting of all automorphisms fixing $Q$. Further, for each non-negative integer $i$, the $i$-th ramification group $\text{Aut}_{Q}^{(i)}(\mathcal{Y})$ at $Q$ is defined by
\begin{equation*}
\text{Aut}_{Q}^{(i)}(\mathcal{Y})=\bigg\{\sigma\in\text{Aut}_{Q}(\mathcal{Y})\,\,:\,\,v_{Q}(\sigma(\mathfrak{t})-\mathfrak{t})\geqslant i+1\bigg\},
\end{equation*}
where $v_{Q}$ is the discrete valuation associated to $Q$, $\mathfrak{t}\in\overline{\mathbb{F}_{q}}(\mathcal{Y})$ is a local parameter at $Q$, and $\text{Aut}_{Q}^{(i)}(\mathcal{Y})\supseteq\text{Aut}_{Q}^{(i+1)}(\mathcal{Y})$ for all $i\geqslant 0$. The following result summarizes properties of these subgroups.

%-------------------------------------------------------------------------------------------------------------------------------------------------------------
\begin{theorem}[\citealp{HKT}, Lemma $11.44$ and Theorem $11.74$]
\label{Theorem: GQ is the semi-direct product of its Sylow p-subgroup and a normal cyclic group}
$\emph{Aut}_{Q}^{(0)}(\mathcal{Y})=\emph{Aut}_{Q}(\mathcal{Y})$. Moreover, $\emph{Aut}_{Q}^{(1)}(\mathcal{Y})$ is the unique Sylow $p$-subgroup of $\emph{Aut}_{Q}(\mathcal{Y})$, and $\emph{Aut}_{Q}(\mathcal{Y})=\emph{Aut}_{Q}^{(1)}(\mathcal{Y})\rtimes \mathbb{H}$, where $\mathbb{H}$ is a cyclic subgroup of $\emph{Aut}_{Q}(\mathcal{Y})$ of order coprime to $p$.
\end{theorem}
%-------------------------------------------------------------------------------------------------------------------------------------------------------------

The following result is used to prove Theorem \ref{Theorem: the automorphism group of XGS}.

%-------------------------------------------------------------------------------------------------------------------------------------------------------------
\begin{theorem}[\citealp{HKT}, Theorem $11.140$]
\label{Theorem: sufficient condition for the stabilizer of a point be the entire group of automorphisms}
Suppose that $g\geqslant 2$, and let $Q\in\mathcal{Y}$ be a point satisfying
\begin{equation*}
|\emph{Aut}_{Q}^{(1)}(\mathcal{Y})|> 2g+1.
\end{equation*}
Then, one of the following cases occurs:
\begin{enumerate}
%------------------------------------------------------------------------------
\item[$1.$] $\emph{Aut}(\mathcal{Y})=\emph{Aut}_{Q}(\mathcal{Y})$.
%------------------------------------------------------------------------------
\item[$2.$] $\mathcal{Y}$ is birationally equivalent to one of the following curves:
\begin{enumerate}
%---------------------------------------
\item[\emph{a)}] the Hermitian curve
\begin{equation}
\label{Equation: the Hermitian curve}
Y^{\ell}+Y=X^{\ell+1}
\end{equation}
where $p\geqslant 2$ and $\ell=p^{k}$.
%---------------------------------------
\item[\emph{b)}] the Suzuki curve, given by the nonsingular model of 
\begin{equation}
\label{Equation: the Suzuki curve}
Y^{\ell}+Y=X^{\ell_{0}}(X^{\ell}+X),
\end{equation}
where $p=2$, $\ell_{0}=2^{k}$, and $\ell=2\ell_{0}^{2}$.
%---------------------------------------
\item[\emph{c)}] the Ree curve, given by the nonsingular model of
\begin{equation}
\label{Equation: the Ree curve}
Y^{\ell^{2}}=[1+(X^{\ell}-X)^{\ell-1}]Y^{\ell}-(X^{\ell}-X)^{\ell-1}Y+X^{\ell}(X^{\ell}-X)^{\ell+3\ell_{0}},
\end{equation}
where $p=3$, $\ell_{0}=3^{k}$, and $\ell=3\ell_{0}^{2}$.
\end{enumerate}
\end{enumerate}
\end{theorem}
%-------------------------------------------------------------------------------------------------------------------------------------------------------------

%%%%%%%%%%%%%%%%%%%%%%%%%%%%%%%%%%%%%%%%%%%%%%%%%%%%%%%%%%
\section{Proof of Theorem \ref{Theorem: the number of Fqn rational points on XGS}}\label{Section: Proof of Theorem 1}

\begin{proof}[\unskip\nopunct]%===========================================================
Considering the notation introduced in Subsection \ref{Subsection: The curve GS}, then Theorem \ref{Theorem: the number of Fqn rational points on XGS} is proved using a sequence of steps.

%/////////////////////////////////////////////////////////////////////////////////////////////////////////////////////////////////////////////////////////////
\subsection{Elementary Abelian $p$-Extension\label{Subsection: Step 1 - elementary abelian p-extension}}

The following result summarizes properties of the extension $\overline{\mathbb{F}_{q}}(x,y)/\overline{\mathbb{F}_{q}}(x)$.

%-------------------------------------------------------------------------------------------------------------------------------------------------------------
\begin{proposition}[\citealp{GS}, Propositions $1.1$ and $1.2$]
\label{Proposition: description of elementary abelian p-extensions containing finite fields}
$\overline{\mathbb{F}_{q}}(x,y)/\overline{\mathbb{F}_{q}}(x)$ is an elementary abelian $p$-extension of degree $q$. The set of intermediate fields $\overline{\mathbb{F}_{q}}(x)\subseteq E\subseteq\overline{\mathbb{F}_{q}}(x,y)$ such that $[E:\overline{\mathbb{F}_{q}}(x)]=p$ is given by
\begin{equation*}
\bigg\{E_{\alpha}\,\,:\,\,\alpha\in\mathbb{F}_{q}^{\ast}\bigg\},
\end{equation*}
where for each $\alpha\in\mathbb{F}_{q}^{\ast}$, $E_{\alpha}$ is the intermediate field of $\overline{\mathbb{F}_{q}}(x,y)/\overline{\mathbb{F}_{q}}(x)$ given by
\begin{equation*}
E_{\alpha}\coloneqq\overline{\mathbb{F}_{q}}(x,y_{\alpha}),
\end{equation*}
with $y_{\alpha}\coloneqq(\alpha y)^{p^{m-1}}+(\alpha y)^{p^{m-2}}+\cdots+(\alpha y)^{p}+\alpha y$ satisfying the equation $y_{\alpha}^p-y_{\alpha}=\alpha x^{q_{0}}(x^{q}-x)$. Further, $E_{\alpha_{1}}=E_{\alpha_{2}}$ if and only if $\alpha_{1}=\alpha\alpha_{2}$ for some $\alpha\in\mathbb{F}_{p}^{\ast}$. Therefore, there are exactly 
$
\frac{q-1}{p-1}
$
intermediate fields $E$ of $\overline{\mathbb{F}_{q}}(x,y)/\overline{\mathbb{F}_{q}}(x)$ such that $[E:\overline{\mathbb{F}_{q}}(x)]=p$.
\end{proposition}
%-------------------------------------------------------------------------------------------------------------------------------------------------------------

Based on Proposition \ref{Proposition: description of elementary abelian p-extensions containing finite fields}, for each $\alpha\in\mathbb{F}_{q}^{\ast}$, consider the plane model $(\mathcal{F}_{\alpha},(x,y_{\alpha}))$ for $E_{\alpha}$, where $\mathcal{F}_{\alpha}$ is the geometrically irreducible plane curve defined over $\mathbb{F}_{q}$ given in affine coordinates by\footnote{The geometric irreducibility of $\mathcal{F}_{\alpha}$ for each $\alpha\in\mathbb{F}_{q}^{\ast}$ follows from the geometric irreducibility of $\mathcal{G}_{\mathcal{S}}$ \cite[Lemma $1.3$]{GS}.}
\begin{equation*}
\mathcal{F}_{\alpha}:Y^{p}-Y=\alpha X^{q_{0}}(X^{q}-X).
\end{equation*}

Also, let $\mathcal{X}_{\alpha}$ be the nonsingular model of $\mathcal{F}_{\alpha}$. If $\bigg\{\alpha_{i}\,\,:\,\,i=1,\ldots,\frac{q-1}{p-1}\bigg\}\subseteq\mathbb{F}_{q}^{\ast}$ is such that
\begin{equation*}
\bigg\{E_{\alpha}\,\,:\,\,\alpha\in\mathbb{F}_{q}^{\ast}\bigg\}=\bigg\{E_{\alpha_{i}}\,\,:\,\,i=1,\ldots,\frac{q-1}{p-1}\bigg\},
\end{equation*}
then the following result relates the L-polynomial of $\mathcal{X}_{\mathcal{G}_{\mathcal{S}}}$ to the L-polynomials of $\mathcal{X}_{\alpha_{i}}$, for $i=1,\ldots,\frac{q-1}{p-1}$. Thus the number of $\mathbb{F}_{q^{n}}$-rational points on $\mathcal{X}_{\mathcal{G}_{\mathcal{S}}}$ can be expressed as a function of the number of $\mathbb{F}_{q^{n}}$-rational points on $\mathcal{X}_{\alpha_{i}}$, with $i=1,\ldots,\frac{q-1}{p-1}$, for each positive integer $n$.

%-------------------------------------------------------------------------------------------------------------------------------------------------------------
\begin{proposition}[\citealp{DSV}, Corollary $6.7$]
\label{Proposition: relation between intermediate fields and the number of rational points in elementary abelian p-extensions}
Considering the previous notation, then $$\displaystyle{L_{\mathcal{X}_{\mathcal{G}_{\mathcal{S}}}/\mathbb{F}_{q}}(T)=\prod_{i=1}^{\frac{q-1}{p-1}}L_{\mathcal{X}_{\alpha_{i}}/\mathbb{F}_{q}}(T)}.$$ In particular,
\begin{equation}
\label{Equation: number of Fqn rational points on XGS as a function of the number of Fqn rational points on Xalphai}
\emph{N}_{q^{n}}(\mathcal{X}_{\mathcal{G}_{\mathcal{S}}})-(q^{n}+1)=\sum_{i=1}^{\frac{q-1}{p-1}}\bigg(\emph{N}_{q^{n}}(\mathcal{X}_{\alpha_{i}})-(q^{n}+1)\bigg)
\end{equation}
for each positive integer $n$.
\end{proposition}
%-------------------------------------------------------------------------------------------------------------------------------------------------------------

The following lemmas describe the number $\text{N}_{q^{n}}(\mathcal{X}_{\mathcal{G}_{\mathcal{S}}})$ for each positive integer $n$.

%-------------------------------------------------------------------------------------------------------------------------------------------------------------
\begin{lemma}
\label{Lemma: number of points on XGS as a function of the number of points on X1}
For each $\alpha\in\mathbb{F}_{q}^{\ast}$, $\mathbb{F}_{q}(x,y_{\alpha})$ is $\mathbb{F}_{q}$-isomorphic to $\mathbb{F}_{q}(x,y_{1})$. In particular, from $\eqref{Equation: number of Fqn rational points on XGS as a function of the number of Fqn rational points on Xalphai}$,
\begin{equation*}
\emph{N}_{q^{n}}(\mathcal{X}_{\mathcal{G}_{\mathcal{S}}})-(q^{n}+1)=\frac{q-1}{p-1}\bigg(\emph{N}_{q^{n}}(\mathcal{X}_{1})-(q^{n}+1)\bigg)
\end{equation*}
for each positive integer $n$.
\end{lemma}
\begin{proof}\renewcommand{\qedsymbol}{$\square$}
Recall that $q=p^{m}$ and $q_{0}=p^{t}$, where $m=2t-1$; that is, $m-t=t-1$. For $\alpha\in\mathbb{F}_{q}^{\ast}$, let $\beta\coloneqq\alpha^{p^{m-t}+p^{m-t-1}+\cdots+p+1}\in \mathbb{F}_{q}$ and $\gamma\coloneqq\text{N}_{\mathbb{F}_{q}/\mathbb{F}_{p}}(\alpha)\in\mathbb{F}_p$. Then,
\begin{eqnarray*}
\beta^{q_0}\beta&=&\bigg(\alpha^{p^{m-t}+p^{m-t-1}+\cdots+p+1}\bigg)^{p^{t}}\alpha^{p^{m-t}+p^{m-t-1}+\cdots+p+1}\\
%&=&\alpha^{p^{m}+p^{m-1}+\cdots+p^{t+1}+p^{t}}\cdot\alpha^{p^{m-t}+p^{m-t-1}+\cdots+p+1}\\
&=&\alpha^{p^{m}+p^{m-1}+\cdots+p^{t+1}+p^{t}}\cdot\alpha^{p^{t-1}+p^{t-2}+\cdots+p+1}\\
&=&\alpha^{p^{m}}\alpha^{p^{m-1}+\cdots+p+1}\\
&=&\alpha\gamma,
\end{eqnarray*}
and $\mathbb{F}_{q}(x,y_{\alpha})=\mathbb{F}_{q}(\beta x,\gamma y_{\alpha})$, with 
\begin{eqnarray*}
(\beta x)^{q_{0}}((\beta x)^{q}-(\beta x))&=&\beta^{q_{0}}\beta x^{q_{0}}(x^{q}-x)\\
&=&\alpha\gamma x^{q_{0}}(x^{q}-x)\\
%&=&\gamma (y_{\alpha}^{p}-y_{\alpha})\\
&=&(\gamma y_{\alpha})^{p}-(\gamma y_{\alpha}).
\end{eqnarray*}
Therefore, 
\begin{equation*}
\begin{array}{cccc}
\mathbb{F}_{q}(x,y_{\alpha})&\rightarrow&\mathbb{F}_{q}(x,y_{1})\\
\frac{A(\beta x,\gamma y_{\alpha})}{B(\beta x,\gamma y_{\alpha})}&\mapsto&\frac{A(x,y_{1})}{B(x,y_{1})}
\end{array}
\end{equation*}
is an $\mathbb{F}_{q}$-isomorphism between $\mathbb{F}_{q}(x,y_{\alpha})$ and $\mathbb{F}_{q}(x,y_{1})$, where $A(X,Y),B(X,Y)\in\mathbb{F}_{q}[X,Y]$ and $B(\beta x,\gamma y_{\alpha})\neq 0$.
\end{proof}
%-------------------------------------------------------------------------------------------------------------------------------------------------------------

%-------------------------------------------------------------------------------------------------------------------------------------------------------------
\begin{lemma}
\label{Lemma: number of points on XGS as a function of the number of points on tilX1}
For each $\alpha\in\mathbb{F}_{q}^{\ast}$, $E_{1}=\overline{\mathbb{F}_{q}}(x,z_{1})$ and $\mathbb{F}_{q}(x,y_{1})=\mathbb{F}_{q}(x,z_{1})$, where 
\begin{equation*}
z_{1}\coloneqq y_{1}-x^{\frac{q}{p}}x^{\frac{q_0}{p}}-x^{\frac{q}{p^{2}}}x^{\frac{q_0}{p^{2}}}-\cdots -x^{\frac{q}{p^{t-1}}}x^{\frac{q_0}{p^{t-1}}} - x^{\frac{q}{q_{0}}}x^{\frac{q_{0}}{q_{0}}}\in\mathbb{F}_{q}(x,y_{1})
\end{equation*}
satisfies $z_{1}^{p}-z_{1}=x(x^{\frac{q}{q_{0}}}-x^{q_{0}})$. In particular, if $\tilde{\mathcal{F}}_{1}$ is the geometrically irreducible plane curve given in affine coordinates by
\begin{equation*}
\tilde{\mathcal{F}}_{1}:Y^{p}-Y=X(X^{\frac{q}{q_{0}}}-X^{q_{0}}),
\end{equation*}
and $\tilde{\mathcal{X}}_{1}$ is its nonsingular model, then $(\tilde{\mathcal{F}}_{1},(x,z_{1}))$ is another plane model for $E_{1}$, and 
\begin{equation*}
\emph{N}_{q^{n}}(\mathcal{X}_{\mathcal{G}_{\mathcal{S}}})-(q^{n}+1)=\frac{q-1}{p-1}\bigg(\emph{N}_{q^{n}}(\tilde{\mathcal{X}}_{1})-(q^{n}+1)\bigg)
\end{equation*}
for each positive integer $n$.
\end{lemma}
\begin{proof}\renewcommand{\qedsymbol}{$\square$}
The equalities $E_{1}=\overline{\mathbb{F}_{q}}(x,z_{1})$ and $\mathbb{F}_{q}(x,y_{1})=\mathbb{F}_{q}(x,z_{1})$ follow immediately from the definition of $z_{1}$. Further, 
\begin{eqnarray*}
z_{1}^p-z_{1}&=&y_{1}^p-y_{1}-x^{q}x^{q_0}+x^{\frac{q}{q_0}}x^{\frac{q_0}{q_0}}\\
&=&x^{q_{0}}(x^{q}-x)-x^{q}x^{q_0}+x^{\frac{q}{q_0}}x\\
&=&x(x^{\frac{q}{q_0}}-x^{q_0})
\end{eqnarray*}
as desired, and the other statements follow from the geometric irreducibility of $\tilde{\mathcal{F}}_{1}$ given in Proposition \ref{Proposition: properties of the plane curve FR}.
\end{proof}
%-------------------------------------------------------------------------------------------------------------------------------------------------------------

%/////////////////////////////////////////////////////////////////////////////////////////////////////////////////////////////////////////////////////////////
\subsection{The Number of $\mathbb{F}_{q^{n}}$-Rational Points on $\tilde{\mathcal{F}}_{1}:Y^{p}-Y=X(X^{\frac{q}{q_{0}}}-X^{q_{0}})$}\label{Subsection: Step 2 - The number of Fqn-rational points on tilF1}

Considering the notation as in Subsection \ref{Subsection: On the number of Fqn-rational points of Yp-Y=XR(X)}, let $\tilde{\mathcal{F}}_{1}=\mathcal{F}_{R}$, where $$R(X)=X^{\frac{q}{q_{0}}}-X^{q_{0}}.$$

To apply Proposition \ref{Proposition: number of solutions as a function of the dimension}, let us determine the dimension of $W_{R}^{(n)}$ over $\mathbb{F}_{p}$ based on the characterization provided in Proposition \ref{Proposition: characterization of WRn}. For this, we first note that
\begin{equation*}
E_{R}(T)=(T-T^{q})^{p}-(T-T^{q}).
\end{equation*}

Thus the following statements are equivalent for an element $\alpha\in\mathbb{F}_{q^{n}}$:
\begin{enumerate}
\item[1.] $\alpha\in W_{R}^{(n)}$
\item[2.] $E_{R}(\alpha)=0$
\item[3.] $\alpha^{q}-\alpha\in\mathbb{F}_{p}$.
\end{enumerate}

The following lemma is a consequence of the results in \cite{CH} and \cite{LIANG}.

%-------------------------------------------------------------------------------------------------------------------------------------------------------------
\begin{lemma}
\label{Lemma: the dimension of WRn}
Let $\beta\in\mathbb{F}_{p}$ be fixed. Then,
\begin{equation*}
\#\bigg\{\alpha\in\mathbb{F}_{q^n}\,\,:\,\,\alpha^{q}-\alpha-\beta=0\bigg\}=\left\{\begin{array}{ll}
q,& \text{ if }n\beta=0\\
0,& \text{ otherwise}.
\end{array}\right.
\end{equation*}
In particular, the splitting field of $E_{R}(T)$ is $\mathbb{F}_{q^{p}}$,
\begin{equation*}
\#W_{R}^{(n)}=\left\{\begin{array}{cl}
pq,& \text{ if }p\mid n\\
q,& \text{ otherwise},
\end{array}\right.
\end{equation*}
and 
\begin{equation*}
\dim_{\mathbb{F}_{p}}(W_{R}^{(n)})=\left\{\begin{array}{cl}
m+1,& \text{ if }p\mid n\\
m,& \text{ otherwise}.
\end{array}\right.
\end{equation*}
\end{lemma}
%-------------------------------------------------------------------------------------------------------------------------------------------------------------

Recall that $m$ is odd. Therefore, since Proposition \ref{Proposition: properties of the plane curve FR} provides that the genus of $\tilde{\mathcal{X}}_{1}$ is given by
$$
g(\tilde{\mathcal{X}}_{1})\coloneqq \frac{p^{t}(p-1)}{2},
$$ 
as a consequence of  Proposition \ref{Proposition: number of solutions as a function of the dimension}, Lemma \ref{Lemma: the dimension of WRn}, and the considerations in Subsection \ref{Subsection: On the number of Fqn-rational points of Yp-Y=XR(X)}, the following holds.

%-------------------------------------------------------------------------------------------------------------------------------------------------------------
\begin{lemma}
\label{Lemma: number of rational points on tilX1 without determining the sign}
The number of $\mathbb{F}_{q^{n}}$-rational points on $\tilde{\mathcal{X}}_{1}$ can be described as follows.
\begin{enumerate}
%------------------------------------------------------------------------------
\item[$1.$] If $p\mid n$, then 
\begin{equation}
\label{Equation: number of rational points on tilX1 without the determination of the sign when p mid n}
\emph{N}_{q^{n}}(\tilde{\mathcal{X}}_{1})=\left\{\begin{array}{cl}
q^n+1,& \text{ if }n\text{ is odd}\\
q^n+1\pm 2g(\tilde{\mathcal{X}}_{1})q^{n/2},& \text{ if }n\text{ is even}.
\end{array}\right.
\end{equation}
%------------------------------------------------------------------------------
\item[$2.$] If $p\nmid n$, then 
\begin{equation}
\label{Equation: number of rational points on tilX1 without the determination of the sign when p nmid n}
\emph{N}_{q^{n}}(\tilde{\mathcal{X}}_{1})=\left\{\begin{array}{cl}
q^n+1\pm 2g(\tilde{\mathcal{X}}_{1})q^{n/2}p^{-1/2},& \text{ if }n\text{ is odd}\\
q^n+1,& \text{ if }n\text{ is even}.
\end{array}\right.
\end{equation}
\end{enumerate}
\noindent In particular, 
$
\emph{N}_{q}(\tilde{\mathcal{X}}_{1})=qp+1,
$
and $\tilde{\mathcal{X}}_{1}$ is $\mathbb{F}_{q^{n}}$-maximal or minimal if and only if $p\mid n$ and $n$ is even.
\end{lemma}
%-------------------------------------------------------------------------------------------------------------------------------------------------------------

Therefore, to determine the sign $\pm$ in \eqref{Equation: number of rational points on tilX1 without the determination of the sign when p mid n} and \eqref{Equation: number of rational points on tilX1 without the determination of the sign when p nmid n}, the analysis is separated in two cases.

%++++++++++++++++++++++++++++++++++++++++++++++++++++++++++++++++++++++++++++++
\subsubsection{The Case $p\mid n$}

The following result is used to determine the sign in \eqref{Equation: number of rational points on tilX1 without the determination of the sign when p mid n}.

%-------------------------------------------------------------------------------------------------------------------------------------------------------------
\begin{lemma}[\citealp{BOUWetal}, Comments on page $105$ and Theorem $7.4$]
\label{Lemma: tilX1 covers Xa}
There exist $\beta\in\mathbb{F}_{q^{p}}$ and an $\mathbb{F}_{q^{p}}$-rational morphism $\varphi:\tilde{\mathcal{X}}_{1}\rightarrow \mathfrak{X}_{\beta}$, where $\mathfrak{X}_{\beta}$ is the nonsingular model defined over $\mathbb{F}_{q^{p}}$ of
\begin{equation*}
Y^{p}-Y=\beta X^{2}.
\end{equation*}
\end{lemma}
%-------------------------------------------------------------------------------------------------------------------------------------------------------------

Based on the previous result, consider $\tilde{\mathcal{X}}_{1}$, $\mathfrak{X}_{\beta}$, and $\varphi$ defined over the extensions $\mathbb{F}_{q^{n}}$ of $\mathbb{F}_{q^{p}}$, where $n$ is even. In particular, $\beta$ is a square in all these extensions, and since $m$ is odd, the following occurs.

%-------------------------------------------------------------------------------------------------------------------------------------------------------------
\begin{proposition}[\citealp{BOUWetal}, Lemma $9.1$]
\label{Proposition: maximality of Xa}
Let $n$ be an even positive integer such that $p\mid n$. Then, $\mathfrak{X}_{\beta}$ is  $\mathbb{F}_{q^{n}}$-maximal if and only if $\upeta((-1)^{n/2})=-1$; that is, if and only if
\begin{equation*}
p\equiv 3\pmod{4} \text{ and } n\equiv2\pmod{4}.
\end{equation*}
\end{proposition}
%-------------------------------------------------------------------------------------------------------------------------------------------------------------

Therefore, from Theorem \ref{Theorem: theorem of Serre}, the sign in \eqref{Equation: number of rational points on tilX1 without the determination of the sign when p mid n} is obtained.

%-------------------------------------------------------------------------------------------------------------------------------------------------------------
\begin{lemma}
\label{Lemma: number of rational point in tilX1 when p|n}
Let $n$ be an even positive integer such that $p\mid n$. The curve $\tilde{\mathcal{X}}_{1}$ is $\mathbb{F}_{q^{n}}$-maximal if and only if $p\equiv 3\pmod{4}$ and $n\equiv2\pmod{4}$. In particular, \eqref{Equation: number of rational points on tilX1 without the determination of the sign when p mid n} can be rewritten in the form
\begin{equation*}
\emph{N}_{q^{n}}(\tilde{\mathcal{X}}_{1})=\left\{\begin{array}{cl}
q^{n}+1,&\text{ if }n\text{ is odd}\\
q^{n}+1-\upeta((-1)^{n/2})\cdot 2g(\tilde{\mathcal{X}}_{1})q^{n/2},& \text{ if }n\text{ is even}.
\end{array}\right.
\end{equation*}
\end{lemma}
%-------------------------------------------------------------------------------------------------------------------------------------------------------------

%++++++++++++++++++++++++++++++++++++++++++++++++++++++++++++++++++++++++++++++
\subsubsection{The Case $p\nmid n$}

Let $s$ be the period of $\tilde{\mathcal{X}}_{1}$ over $\mathbb{F}_{q}$. From Lemmas \ref{Lemma: number of rational points on tilX1 without determining the sign} and \ref{Lemma: number of rational point in tilX1 when p|n}, 
\begin{equation}
\label{Equation: period of tilX1}
s=\left\{\begin{array}{ll}
2p,& \text{ if }\upeta(-1)=1\\
4p,& \text{ if }\upeta(-1)=-1.
\end{array}\right.
\end{equation}

Since $\gcd(n,s)=1$, for each odd positive integer $n$ such that $p\nmid n$, and $\text{N}_{q}(\tilde{\mathcal{X}}_{1})=qp+1$ from Lemma \ref{Lemma: number of rational points on tilX1 without determining the sign}, the sign in \eqref{Equation: number of rational points on tilX1 without the determination of the sign when p nmid n} is obtained via a straightforward application of Theorem \ref{Theorem: Gary McGuire and Yilmaz' theorem}.

%-------------------------------------------------------------------------------------------------------------------------------------------------------------
\begin{lemma}
\label{Lemma: number of rational point in tilX1 when p nmid n}
If $p\nmid n$, then 
\begin{equation*}
\emph{N}_{q^{n}}(\tilde{\mathcal{X}}_1)=\left\{\begin{array}{cl}
q^{n}+1+\upeta((-1)^{(n-1)/2}n)\cdot 2g(\tilde{\mathcal{X}}_1)q^{n/2}p^{-1/2},& \text{ if }n\text{ is odd}\\
q^{n}+1,& \text{ if }n\text{ is even}.
\end{array}\right.
\end{equation*}
\end{lemma}
%-------------------------------------------------------------------------------------------------------------------------------------------------------------

%/////////////////////////////////////////////////////////////////////////////////////////////////////////////////////////////////////////////////////////////
\subsection{Conclusion}\label{Subsection: Step 3 - Conclusion}

Theorem \ref{Theorem: the number of Fqn rational points on XGS} follows from Lemmas \ref{Lemma: number of points on XGS as a function of the number of points on tilX1}, \ref{Lemma: number of rational point in tilX1 when p|n}, and \ref{Lemma: number of rational point in tilX1 when p nmid n}.
\end{proof}

%%%%%%%%%%%%%%%%%%%%%%%%%%%%%%%%%%%%%%%%%%%%%%%%%%%%%%%%%%
\section{Proof of Theorem \ref{Theorem: the L-polynomial of XGS}}\label{Section: Proof of Theorem 2}

\begin{proof}[\unskip\nopunct]%===========================================================
From Proposition \ref{Proposition: relation between intermediate fields and the number of rational points in elementary abelian p-extensions} and Lemma \ref{Lemma: number of points on XGS as a function of the number of points on tilX1},
\begin{equation}
\label{Equation: the L-polynomial in the p-elementary abelian extension}
L_{\mathcal{X}_{\mathcal{G}_{\mathcal{S}}}/\mathbb{F}_{q}}(T)=\bigg(L_{\tilde{\mathcal{X}}_{1}/\mathbb{F}_{q}}(T)\bigg)^{\frac{q-1}{p-1}},
\end{equation}
which shows that $L_{\mathcal{X}_{\mathcal{G}_{\mathcal{S}}}/\mathbb{F}_{q}}(T)$ is essentially determined by $L_{\tilde{\mathcal{X}}_{1}/\mathbb{F}_{q}}(T)$. Considering $M_{\tilde{\mathcal{X}}_{1}/\mathbb{F}_{q}}(T)$ the reciprocal of the polynomial $L_{\tilde{\mathcal{X}}_{1}/\mathbb{F}_{q}}(q^{-1/2}T)\in\mathbb{Q}(p^{1/2})[T]$, then
\begin{equation*}
M_{\tilde{\mathcal{X}}_{1}/\mathbb{F}_{q}}(T)=\prod_{i=1}^{2g(\tilde{\mathcal{X}}_{1})}(T-\xi_{i}),
\end{equation*}
where, for $i=1,\ldots,2g(\tilde{\mathcal{X}}_{1})$, the elements $\xi_{i}\in\mathbb{C}$ satisfy
\begin{equation}
\label{Equation: number of rational points in tilX1 and Weil numbers}
-q^{-n/2}\bigg[\text{N}_{q^{n}}(\tilde{\mathcal{X}}_1)-(q^{n}+1)\bigg]=\sum_{i=1}^{2g(\tilde{\mathcal{X}}_{1})}\xi_{i}^{n}.
\end{equation}

Note that $M_{\tilde{\mathcal{X}}_{1}/\mathbb{F}_{q}}(T)$ is a self-reciprocal polynomial; that is, $$M_{\tilde{\mathcal{X}}_{1}/\mathbb{F}_{q}}(T)=L_{\tilde{\mathcal{X}}_{1}/\mathbb{F}_{q}}(q^{-1/2}T).$$ Thus to determine the L-polynomial $L_{\tilde{\mathcal{X}}_{1}/\mathbb{F}_{q}}(T)$, it is sufficient to describe the polynomial $M_{\tilde{\mathcal{X}}_{1}/\mathbb{F}_{q}}(T)$. Further, since $\tilde{\mathcal{X}}_{1}$ is a supersingular curve with period $s$ given by \eqref{Equation: period of tilX1}, the following holds
\begin{equation}
\label{Equation: period of tilX1 and Weil numbers}
\min\bigg\{n\,\,:\,\,\xi_{i}^{n}=1\text{ for all }i=1,\ldots,2g(\tilde{\mathcal{X}}_{1})\bigg\}=\left\{\begin{array}{ll}
2p, \text{ if }\upeta(-1)=1\\
4p, \text{ if }\upeta(-1)=-1.
\end{array}\right.
\end{equation}

From \eqref{Equation: period of tilX1 and Weil numbers} and Lemma \ref{Lemma: number of rational point in tilX1 when p|n}, 
\begin{equation*}
\xi_{i}^{2}\in\left\{\begin{array}{cl}
\mu_{p}\cup\{1\},& \text{ if }\upeta(-1)=1\\
\mu_{2p}\cup\{-1\},& \text{ if }\upeta(-1)=-1,
\end{array}\right.
\end{equation*}
 and so
\begin{equation*}
\sum\limits_{i=1}^{2g(\tilde{\mathcal{X}}_{1})}\xi_{i}^{2}=r_{0}\bigg(\upeta(-1)\cdot 1\bigg)+r_{1}\bigg(\upeta(-1)\cdot \zeta_{p}\bigg)+\cdots+r_{p-1}\bigg(\upeta(-1)\cdot\zeta_{p}^{p-1}\bigg),
\end{equation*}
with $r_{i}\in\mathbb{N}$ being such that $\displaystyle{\sum_{i=0}^{p-1}r_{i}=2g(\tilde{\mathcal{X}}_{1})=p^{t}(p-1)}$. Also, \eqref{Equation: number of rational points in tilX1 and Weil numbers} and Lemma \ref{Lemma: number of rational point in tilX1 when p nmid n} imply that $$\displaystyle{\sum_{i=1}^{2g(\tilde{\mathcal{X}}_{1})}\xi_{i}^{2}=0},$$ and then $r_{0}=r_{1}=\cdots=r_{p-1}=p^{t-1}(p-1)$. Thus
\begin{equation}
\label{Equation: factorization of polynomial MX1}
M_{\tilde{\mathcal{X}}_{1}/\mathbb{F}_{q}}(T)=\prod\limits_{i=1}^{2g(\tilde{\mathcal{X}}_{1})}(T-\xi_{i})=\prod_{i=0}^{p-1}(T-\lambda_{i})^{m_{i}}(T+\lambda_i)^{n_{i}},
\end{equation}
where $m_{i},n_{i}\in\mathbb{N}$ are such that $m_{i}+n_{i}=p^{t-1}(p-1)$, and $\pm \lambda_{i}$ are the square roots of the $p$ elements in 
\begin{equation*}
\left\{\begin{array}{cl}
\mu_{p}\cup\{1\},& \text{ if }\upeta(-1)=1\\
\mu_{2p}\cup\{-1\},& \text{ if }\upeta(-1)=-1
\end{array}\right.
\end{equation*}
for each $i=0,\ldots,p-1$. Label the elements $\lambda_{i}$ in a way that 
\begin{equation*}
\lambda_{i}=\left\{\begin{array}{rl}
\zeta_{p}^{i},& \text{ if }\upeta(-1)=1\\
\textbf{i}\zeta_{p}^{i},& \text{ if }\upeta(-1)=-1.
\end{array}\right.
\end{equation*}

Furthermore, from \eqref{Equation: number of rational points in tilX1 and Weil numbers} and Lemmas \ref{Lemma: number of rational point in tilX1 when p|n}, \ref{Lemma: number of rational point in tilX1 when p nmid n}, $\displaystyle{\sum_{i=1}^{2g(\tilde{\mathcal{X}}_{1})}\xi_{i}^{p}=0}$ and $$\displaystyle{\sum\limits_{i=1}^{2g(\tilde{\mathcal{X}}_{1})}\xi_{i}=-p^{t-1}(p-1)p^{1/2}}.$$ 

Since \eqref{Equation: period of tilX1 and Weil numbers} and Lemma \ref{Lemma: number of rational point in tilX1 when p|n} imply that 
\begin{equation*}
\xi_{i}^{p}=\left\{\begin{array}{cl}
\pm 1,& \text{ if }\upeta(-1)=1\\
\pm \textbf{i},& \text{ if }\upeta(-1)=-1,
\end{array}\right.
\end{equation*}
by the choice of $\lambda_{i}$ and equation \eqref{Equation: factorization of polynomial MX1},
\begin{eqnarray}
\label{Equation: first relation}
g(\tilde{\mathcal{X}}_{1})&=&\sum\limits_{i=0}^{p-1}m_{i}=\sum\limits_{i=0}^{p-1}n_{i}\\
-p^{t-1}(p-1)p^{1/2}&=&\sum\limits_{i=0}^{p-1}(m_{i}-n_{i})\lambda_{i}.
\end{eqnarray}

From the Quadratic Gauss Sum \cite[Theorem $5.15$]{LN}, 
\begin{equation}
p^{1/2}=\sum_{i=0}^{p-1}\upeta(-i)\lambda_{i}.
\end{equation}
Hence \eqref{Equation: first relation} and $m_{i}+n_{i}=p^{t-1}(p-1)$ yield $n_{0}=m_{0}=\frac{p^{t-1}(p-1)}{2}$, and considering $1\leqslant i \leqslant p-1$,
$$
(m_{i},n_{i})=\left\{\begin{array}{rl}
(0,p^{t-1}(p-1)),  &\text{if } \upeta(-i)=1\\
\text{}&\\
(p^{t-1}(p-1), 0),  &\text{if } \upeta(-i)=-1.
\end{array}\right.
$$
In particular, $M_{\tilde{\mathcal{X}}_{1}/\mathbb{F}_{q}}(T)$ in \eqref{Equation: factorization of polynomial MX1} becomes
\begin{equation*}
\displaystyle{(T^{2}-\upeta(-1))^{\frac{p^{t-1}(p-1)}{2}}\bigg(\prod_{\upeta(i)=1}(T+\upeta(-1)\lambda_{i})\prod_{\upeta(i)=-1}(T-\upeta(-1)\lambda_{i})\bigg)^{p^{t-1}(p-1)}}.
\end{equation*}

Let $\sigma_{i}\in \text{Gal}(\mathbb{Q}(\zeta_{p})/\mathbb{Q})$ be such that $\sigma_{i}(\zeta_{p})=\zeta_{p}^{i}$, with $1\leqslant i\leqslant p-1$. Then, from the Quadratic Gauss Sum
$$
\sigma_{i}(\tilde{p}^{1/2})=\upeta(i)\cdot\tilde{p}^{1/2},
$$
where $\tilde{p}^{1/2}\in\mathbb{Q}(\zeta_{p})$ is defined as in \eqref{Equation: definition of p*1/2}, and  
$$
\sigma_{i}(-\zeta_{p}/\tilde{p}^{1/2})=-\upeta(-i)\lambda_{i}/p^{1/2}=-\upeta(i)\cdot\upeta(-1)\lambda_{i}/p^{1/2}
$$
for each $i\in\{1,\ldots,p-1\}$ yields
\begin{eqnarray*}
&&\prod_{\upeta(i)=1}(p^{1/2}T+\upeta(-1)\lambda_{i})\prod_{\upeta(i)=-1}(p^{1/2}T-\upeta(-1)\lambda_{i})\\
&=&p^{(p-1)/2}\prod_{\upeta(i)=1}(T+\upeta(-1)\lambda_{i}/p^{1/2})\prod_{\upeta(i)=-1}(T-\upeta(-1)\lambda_{i}/p^{1/2})\\
&=&p^{(p-1)/2}\prod_{i=1}^{p-1}\bigg(T-\sigma_{i}(-\zeta_{p}/\tilde{p}^{1/2})\bigg)\\
&=&p^{(p-1)/2}\prod_{\sigma\in \text{Gal}(\mathbb{Q}(\zeta_{p})/\mathbb{Q})}\bigg(T-\sigma(-\zeta_{p}/\tilde{p}^{1/2})\bigg)\\
&=&p^{(p-1)/2}\mathcal{M}_{p}(T),
\end{eqnarray*}
where $\mathcal{M}_{p}(T)$ is the minimal polynomial over $\mathbb{Q}$ of $-\zeta_{p}/\tilde{p}^{1/2}$. Therefore, \eqref{Equation: the L-polynomial over Fq} follows from \eqref{Equation: the L-polynomial in the p-elementary abelian extension} and the equalities $$L_{\tilde{\mathcal{X}}_{1}/\mathbb{F}_{q}}(T)=M_{\tilde{\mathcal{X}}_{1}/\mathbb{F}_{q}}(q^{1/2}T)=M_{\tilde{\mathcal{X}}_{1}/\mathbb{F}_{q}}(p^{t-1}p^{1/2}T).$$ Further, \eqref{Equation: number of rational points in tilX1 and Weil numbers},
$$
M_{\tilde{\mathcal{X}}_{1}/\mathbb{F}_{q^{n}}}(T)=\prod_{i=1}^{2g(\tilde{\mathcal{X}}_{1})}(T-\xi_{i}^{n}),
$$
where $M_{\tilde{\mathcal{X}}_{1}/\mathbb{F}_{q^{n}}}(T)$ is the reciprocal of the polynomial $L_{\tilde{\mathcal{X}}_{1}/\mathbb{F}_{q^{n}}}(q^{-n/2}T)$, and a straightforward calculation give the second part o Theorem \ref{Theorem: the L-polynomial of XGS}.
\end{proof}

%%%%%%%%%%%%%%%%%%%%%%%%%%%%%%%%%%%%%%%%%%%%%%%%%%%%%%%%%%
\section{Proof of Theorem \ref{Theorem: the automorphism group of XGS}}\label{Section: Proof of Theorem 3}

\begin{proof}[\unskip\nopunct]%===========================================================
Consider the notation as in Subsections \ref{Subsection: The curve GS} and \ref{Subsection: Automorphism group}. Let $\text{Aut}_{Q_{\infty}}(\mathcal{X}_{\mathcal{G}_{\mathcal{S}}})$ be the stabilizer of $Q_{\infty}\in\mathcal{X}_{\mathcal{G}_{\mathcal{S}}}$, $\text{Aut}_{Q_{\infty}}^{(1)}(\mathcal{X}_{\mathcal{G}_{\mathcal{S}}})$ be the unique Sylow $p$-subgroup of $\text{Aut}_{Q_{\infty}}(\mathcal{X}_{\mathcal{G}_{\mathcal{S}}})$, and $\mathbb{H}$ be the cyclic complement of $\text{Aut}_{Q_{\infty}}^{(1)}(\mathcal{X}_{\mathcal{G}_{\mathcal{S}}})$ in $\text{Aut}_{Q_{\infty}}(\mathcal{X}_{\mathcal{G}_{\mathcal{S}}})$.

Let $\mathfrak{G}$ be the set of maps on $\overline{\mathbb{F}_{q}}(\mathcal{X}_{\mathcal{G}_{\mathcal{S}}})=\overline{\mathbb{F}_{q}}(x,y)$ given by
\begin{equation*}
(x,y)\mapsto (\alpha x+\beta,\alpha \beta^{q_{0}}x+\alpha^{q_{0}+1}y+\gamma),
\end{equation*}
where $\alpha\in\mathbb{F}_{q}^{\ast}$ and $\beta,\gamma\in\mathbb{F}_{q}$. One can check that $\mathfrak{G}$ is a subgroup of $\text{Aut}(\mathcal{X}_{\mathcal{G}_{\mathcal{S}}})$ of order $q^{2}(q-1)$. Moreover, the following lemma holds.

%-------------------------------------------------------------------------------------------------------------------------------------------------------------
\begin{lemma}
\label{Lemma: G is equal to the stabilizer of Qinfty}
$\mathfrak{G}=\emph{Aut}_{Q_{\infty}}(\mathcal{X}_{\mathcal{G}_{\mathcal{S}}})$.
\end{lemma}
\begin{proof}\renewcommand{\qedsymbol}{$\square$}
The inclusion $\mathfrak{G}\subseteq\text{Aut}_{Q_{\infty}}(\mathcal{X}_{\mathcal{G}_{\mathcal{S}}})$ is straightforward. To show the other inclusion, let $\sigma\in \text{Aut}_{Q_{\infty}}(\mathcal{X}_{\mathcal{G}_{\mathcal{S}}})$. From Proposition \ref{Proposition: some Riemann-Roch spaces}, $\{1,x\}$ and $\{1,x,y\}$ are bases for the Riemann--Roch spaces $\mathscr{L}(qQ_{\infty})$ and $\mathscr{L}((q+q_0)Q_{\infty})$, respectively. Therefore, $\sigma(x)\in\mathscr{L}(qQ_{\infty})$, $\sigma(y)\in\mathscr{L}((q+q_0)Q_{\infty})$, and
\begin{equation*}
\sigma(x)=\alpha x+\gamma\text{ and }\sigma(y)=\alpha_{1}x+\beta_{1}y+\gamma_{1}
\end{equation*}
for some $\alpha,\alpha_{1},\beta_{1},\gamma,\gamma_{1}\in\overline{\mathbb{F}_{q}}$, with $\alpha\beta_{1}\neq 0$. Since $\sigma$ is an $\overline{\mathbb{F}_{q}}$-automorphism,
\begin{equation*}
0=\sigma(0)=\sigma(y^{q}-y-x^{q_{0}}(x^{q}-x))=\sigma(y)^{q}-\sigma(y)-\sigma(x)^{q_{0}}(\sigma(x)^{q}-\sigma(x)),
\end{equation*}
and thus
\begin{equation*}
\left.\begin{array}{c}
(\alpha_{1}X+\beta_{1}Y+\gamma_{1})^{q}-(\alpha_{1}X+\beta_{1}Y+\gamma_{1})-(\alpha X+\gamma)^{q_{0}}((\alpha X+\gamma)^{q}-(\alpha X+\gamma))\\
=\delta(Y^{q}-Y-X^{q_{0}}(X^{q}-X))
\end{array}\right.
\end{equation*}
for some $\delta \in \overline{\mathbb{F}_q}^{\ast}$. Therefore, comparing the coefficients on both sides of the previous equality, the following is obtained
\begin{itemize}
%------------------------------------------------------------------------------
\item $\alpha\in\mathbb{F}_{q}^{\ast}$
%------------------------------------------------------------------------------
\item $\gamma\in\mathbb{F}_{q}$
%------------------------------------------------------------------------------
\item $\alpha_{1}=\alpha\gamma^{q_{0}}\in\mathbb{F}_{q}$
%------------------------------------------------------------------------------
\item $\beta_{1}=\alpha^{q_{0}+1}$
%------------------------------------------------------------------------------
\item $\gamma_{1}\in\mathbb{F}_{q}$,
\end{itemize}
which completes the proof.
\end{proof}
%-------------------------------------------------------------------------------------------------------------------------------------------------------------

Consider the subgroup of $\text{Aut}_{Q_{\infty}}(\mathcal{X}_{\mathcal{G}_{\mathcal{S}}})$ consisting of the maps
\begin{equation*}
(x,y)\mapsto (x+\beta,\beta^{q_{0}}x+y+\gamma),
\end{equation*}
where $\beta,\gamma\in\mathbb{F}_{q}$, which is a Sylow $p$-subgroup of $\text{Aut}_{Q_{\infty}}(\mathcal{X}_{\mathcal{G}_{\mathcal{S}}})$ (of order $q^{2}$). From Theorem \ref{Theorem: GQ is the semi-direct product of its Sylow p-subgroup and a normal cyclic group}, this describes the subgroup $\text{Aut}_{Q_{\infty}}^{(1)}(\mathcal{X}_{\mathcal{G}_{\mathcal{S}}})$. Also, the cyclic complement $\mathbb{H}$ of $\text{Aut}_{Q_{\infty}}^{(1)}(\mathcal{X}_{\mathcal{G}_{\mathcal{S}}})$ in $\text{Aut}_{Q_{\infty}}(\mathcal{X}_{\mathcal{G}_{\mathcal{S}}})$ can be given by the maps
\begin{equation*}
(x,y)\mapsto (\alpha x,\alpha^{q_{0}+1}y),
\end{equation*}
where $\alpha\in\mathbb{F}_{q}^{\ast}$, which has order $q-1$, and
\begin{equation*}
\text{Aut}_{Q_{\infty}}(\mathcal{X}_{\mathcal{G}_{\mathcal{S}}})=\text{Aut}_{Q_{\infty}}^{(1)}(\mathcal{X}_{\mathcal{G}_{\mathcal{S}}})\rtimes\mathbb{H}.
\end{equation*}

Finally, from Theorem \ref{Theorem: sufficient condition for the stabilizer of a point be the entire group of automorphisms}, $\text{Aut}(\mathcal{X}_{\mathcal{G}_{\mathcal{S}}})=\text{Aut}_{Q_{\infty}}(\mathcal{X}_{\mathcal{G}_{\mathcal{S}}})$, which completes the proof. Indeed $|\text{Aut}_{Q_{\infty}}^{(1)}(\mathcal{X}_{\mathcal{G}_{\mathcal{S}}})|=q^{2}>q_{0}(q-1)+1=2g+1$, where $g$ is the genus of $\mathcal{X}_{\mathcal{G}_{\mathcal{S}}}$, and also, since $p\neq 2$, a comparison of genus (when $t>1$) and inflection points (when $m=t=1$) shows that $\mathcal{X}_{\mathcal{G}_{\mathcal{S}}}$ is not birationally equivalent to any of the curves \eqref{Equation: the Hermitian curve}, \eqref{Equation: the Suzuki curve}, and \eqref{Equation: the Ree curve}.
\end{proof}%=======================================================================

%%%%%%%%%%%%%%%%%%%%%%%%%%%%%%%%%%%%%%%%%%%%%%%%%%%%%%%%%%
\section{Applications, Examples, and Remarks}\label{Section: Applications, Examples, and Remarks}

\noindent The following result regarding Hilbert class fields of curves can be found in \cite[p. 367]{ROSEN} and \cite[Ch. VI, Sect. 2(8)]{SERRE}.

%-------------------------------------------------------------------------------------------------------------------------------------------------------------
\begin{theorem}
\label{Theorem: main result of [34]}
Let $\mathcal{Y}$ be a curve defined over $\mathbb{F}_{q}$ of genus $g(\mathcal{Y})\geqslant 2$. Assume that $\mathcal{Y}$ has a rational point $P_{0}\in\mathcal{Y}(\mathbb{F}_{q})$, and suppose that $\mathcal{Y}$ is embedded in its Jacobian $J_{\mathcal{Y}}$ by considering
$$
P_{0}\mapsto 0.
$$
If $\langle \mathcal{Y}(\mathbb{F}_{q}) \rangle \varsubsetneq J_{\mathcal{Y}}(\mathbb{F}_{q})$, then, for each divisor $i$ of $[J_{\mathcal{Y}}(\mathbb{F}_{q}):\langle \mathcal{Y}(\mathbb{F}_{q}) \rangle]$, there exists an \'{e}tale cover $\mathcal{Z}$ of $\mathcal{Y}$ of degree $i$ and genus $g(\mathcal{Z})=i\cdot(g(\mathcal{Y})-1)+1$, which satisfies $\emph{N}_{q}(\mathcal{Z})\geqslant i\cdot\emph{N}_{q}(\mathcal{Y})$.
\end{theorem}
%-------------------------------------------------------------------------------------------------------------------------------------------------------------

The following example remarkably illustrates Theorem \ref{Theorem: main result of [34]}.

%-------------------------------------------------------------------------------------------------------------------------------------------------------------
\begin{example}
\begin{proof}[\unskip\nopunct]%===========================================================
Let $p=7$ and $t=1$, and consider $\mathcal{X}_{\mathcal{G}_{\mathcal{S}}}$ embedded in $J_{\mathcal{X}_{\mathcal{G}_{\mathcal{S}}}}$ via
$$
Q_{\infty}\mapsto 0.
$$

Using Magma \cite{MAGMA}, it is possible to check that 
\begin{eqnarray*}
J_{\mathcal{X}_{\mathcal{G}_{\mathcal{S}}}}(\mathbb{F}_{7})&=&\frac{\mathbb{Z}}{1822\mathbb{Z}}+\frac{\mathbb{Z}}{1822\mathbb{Z}}+\frac{\mathbb{Z}}{1822\mathbb{Z}}+\frac{\mathbb{Z}}{1822\mathbb{Z}}+\frac{\mathbb{Z}}{3644\mathbb{Z}}+\frac{\mathbb{Z}}{3644\mathbb{Z}}+\frac{\mathbb{Z}}{3644\mathbb{Z}}\\
&=&\frac{\mathbb{Z}}{1822\mathbb{Z}}+\frac{\mathbb{Z}}{1822\mathbb{Z}}+\frac{\mathbb{Z}}{1822\mathbb{Z}}+\frac{\mathbb{Z}}{1822\mathbb{Z}}+\langle a_{1}\rangle+\langle a_{2}\rangle+\langle a_{3}\rangle,
\end{eqnarray*}
where $a_{i}$ has order $3644=4\cdot 911$ for all $i=1,2,3$, and that the order of each $\mathbb{F}_{7}$-rational point of $\mathcal{X}_{\mathcal{G}_{\mathcal{S}}}$ is $1822=2\cdot 911$. In particular, the order of any element of $\langle \mathcal{X}_{\mathcal{G}_{\mathcal{S}}}(\mathbb{F}_{7}) \rangle$ divides $1822$, which shows that $a_{i}\notin\langle \mathcal{X}_{\mathcal{G}_{\mathcal{S}}}(\mathbb{F}_{7}) \rangle$ for each $i=1,2,3$, and $[J_{\mathcal{X}_{\mathcal{G}_{\mathcal{S}}}}(\mathbb{F}_{7}):\langle \mathcal{X}_{\mathcal{G}_{\mathcal{S}}}(\mathbb{F}_{7}) \rangle]>1$. Further, $911a_{i}+\langle \mathcal{X}_{\mathcal{G}_{\mathcal{S}}}(\mathbb{F}_{7}) \rangle$ has order $2$ or $4$ in $$\frac{J_{\mathcal{X}_{\mathcal{G}_{\mathcal{S}}}}(\mathbb{F}_{7})}{\langle \mathcal{X}_{\mathcal{G}_{\mathcal{S}}}(\mathbb{F}_{7}) \rangle}$$ for each $i=1,2,3$, and so these elements generate a subgroup of order $8$. 

From Theorem \ref{Theorem: main result of [34]}, there exists an \'{e}tale cover $\mathcal{Z}^{(8)}_{\mathcal{X}_{\mathcal{G}_{\mathcal{S}}}}$ (resp. $\mathcal{Z}^{(4)}_{\mathcal{X}_{\mathcal{G}_{\mathcal{S}}}}$) of $\mathcal{X}_{\mathcal{G}_{\mathcal{S}}}$ with genus $161=8\cdot(g(\mathcal{\mathcal{X}_{\mathcal{G}_{\mathcal{S}}}})-1)+1$ (resp. $81=4\cdot(g(\mathcal{\mathcal{X}_{\mathcal{G}_{\mathcal{S}}}})-1)+1$) and at least $400=8\cdot\text{N}_{7}(\mathcal{\mathcal{X}_{\mathcal{G}_{\mathcal{S}}}})$ (resp. $200=4\cdot\text{N}_{7}(\mathcal{\mathcal{X}_{\mathcal{G}_{\mathcal{S}}}})$) rational points. We point out here that from Oesterl\'{e}'s bound, a curve of genus $161$ (resp. $81$) defined over $\mathbb{F}_{7}$ has at most $410$ (resp. $226$) $\mathbb{F}_{7}$-rational points. 

Moreover, also from Theorem \ref{Theorem: main result of [34]}, there exists an \'{e}tale cover $\mathcal{Z}^{(2)}_{\mathcal{X}_{\mathcal{G}_{\mathcal{S}}}}$ of $\mathcal{X}_{\mathcal{G}_{\mathcal{S}}}$ with genus $41=2\cdot(g(\mathcal{\mathcal{X}_{\mathcal{G}_{\mathcal{S}}}})-1)+1$ and at least $100=2\cdot\text{N}_{7}(\mathcal{\mathcal{X}_{\mathcal{G}_{\mathcal{S}}}})$ rational points, which shows that
\begin{equation*}
\text{N}_{41}(7)\coloneqq\max\Bigg\{\text{N}_{7}(\mathcal{Y})\,\,:\,\,\mathcal{Y} \text{ is a curve of genus }41\text{ defined over }\mathbb{F}_{7}\Bigg\}\geqslant 100
\end{equation*}
and gives a new entry in \cite{manYPoints}.
\end{proof}
\end{example}
%-------------------------------------------------------------------------------------------------------------------------------------------------------------

Deciding whether or not the set of $\mathbb{F}_{q}$-rational points on a curve $\mathcal{Y}$ defined over $\mathbb{F}_{q}$ generates the group $J_{\mathcal{Y}}(\mathbb{F}_{q})$ is not an easy task. However, in some cases this can be done using information on the number of $\mathbb{F}_{q}$-rational points on $\mathcal{Y}$ and on its L-polynomial. For instance, the following result can be found in \cite{V}. 

%-------------------------------------------------------------------------------------------------------------------------------------------------------------
\begin{theorem}[Voloch]
\label{Theorem: corollary of the main result of [34]}
Let $\mathcal{Y}$ be a curve defined over $\mathbb{F}_{q}$ of genus $g(\mathcal{Y})\geqslant 2$, and assume that $\mathcal{Y}$ has an $\mathbb{F}_{q}$-rational point. If $\emph{N}_{q}(\mathcal{Y})=\emph{N}_{q^{n}}(\mathcal{Y})$ and $L_{\mathcal{Y}/\mathbb{F}_{q}}(1)<L_{\mathcal{Y}/\mathbb{F}_{q^{n}}}(1)$, then, for each divisor $i$ of $L_{\mathcal{Y}/\mathbb{F}_{q^{n}}}(1)/L_{\mathcal{Y}/\mathbb{F}_{q}}(1)$, there exists an \'{e}tale cover $\mathcal{Z}$ of $\mathcal{Y}$ of degree $i$ and genus $g(\mathcal{Z})=i\cdot(g(\mathcal{Y})-1)+1$, which satisfies $\emph{N}_{q^{n}}(\mathcal{Z})\geqslant i\cdot\emph{N}_{q^{n}}(\mathcal{Y})$. 
\end{theorem}
%\begin{proof}
%It is well known that $\#J_{\mathcal{X}_{\mathcal{G}_{\mathcal{S}}}}(\mathbb{F}_{q^{n}})=L_{\mathcal{X}_{\mathcal{G}_{\mathcal{S}}}/\mathbb{F}_{q^{n}}}(1)$ \cite[Theorem $9.70$]{HKT}. In particular, the hypotheses $\text{N}_{q}(\mathcal{Y})=\text{N}_{q^{n}}(\mathcal{Y})$ and $L_{\mathcal{Y}/\mathbb{F}_{q}}(1)<L_{\mathcal{Y}/\mathbb{F}_{q^{n}}}(1)$ imply that $\langle \mathcal{Y}(\mathbb{F}_{q^{n}}) \rangle=\langle \mathcal{Y}(\mathbb{F}_{q})\rangle \subseteq J_{\mathcal{Y}}(\mathbb{F}_{q}) \varsubsetneq J_{\mathcal{Y}}(\mathbb{F}_{q^{n}})$, and thus the result follows from Theorem \ref{Theorem: main result of [34]}.
%\end{proof}
%-------------------------------------------------------------------------------------------------------------------------------------------------------------

As a consequence of the previous result, the following occurs.

%-------------------------------------------------------------------------------------------------------------------------------------------------------------
\begin{corollary}
\label{Corollary: a way to decide when the rational points generate the class group}
For each divisor $i$ of $L_{\mathcal{X}_{\mathcal{G}_{\mathcal{S}}}/\mathbb{F}_{q^{2}}}(1)/L_{\mathcal{X}_{\mathcal{G}_{\mathcal{S}}}/\mathbb{F}_{q}}(1)$, there exists an \'{e}tale cover of $\mathcal{X}_{\mathcal{G}_{\mathcal{S}}}$ of degree $i$ with at least $i(q^{2}+1)$ rational points over $\mathbb{F}_{q^{2}}$ and genus $i(g-1)+1$. Moreover, if $p>3$ and $p\equiv -1\pmod{3}$, then for each divisor $i$ of $L_{\mathcal{X}_{\mathcal{G}_{\mathcal{S}}}/\mathbb{F}_{q^{3}}}(1)/L_{\mathcal{X}_{\mathcal{G}_{\mathcal{S}}}/\mathbb{F}_{q}}(1)$, there exists an \'{e}tale cover of $\mathcal{X}_{\mathcal{G}_{\mathcal{S}}}$ of degree $i$ with at least $i(q^{2}+1)$ rational points over $\mathbb{F}_{q^{3}}$ and genus $i(g-1)+1$.
\end{corollary}
%-------------------------------------------------------------------------------------------------------------------------------------------------------------

To prove Corollary \ref{Corollary: a way to decide when the rational points generate the class group}, the following lemma, whose proof is straightforward, is used.

%-------------------------------------------------------------------------------------------------------------------------------------------------------------
\begin{lemma}
\label{Lemma: other properties on the L-polynomial}
Consider the notation as in Section \ref{Section: Proof of Theorem 2}. Then, $$p^{p-1}\mathcal{M}_{p}(T)\mathcal{M}_{p}(-T)=\Phi_{p}(\upeta(-1)\cdot pT^{2}),$$ and $\mathcal{M}_{p}(T)$ corresponds to the irreducible factor of $\Phi_{p}(\upeta(-1)\cdot pT^{2})$ with positive linear coefficient. In particular,
\begin{equation*}
L_{\mathcal{X}_{\mathcal{G}_{\mathcal{S}}}/\mathbb{F}_{q}}(1)<L_{\mathcal{X}_{\mathcal{G}_{\mathcal{S}}}/\mathbb{F}_{q^{n}}}(1)
\end{equation*}
for each positive integer $n$.
\end{lemma}
%-------------------------------------------------------------------------------------------------------------------------------------------------------------

%-------------------------------------------------------------------------------------------------------------------------------------------------------------
\begin{proof}[Proof of Corollary \ref{Corollary: a way to decide when the rational points generate the class group}]
From Theorem \ref{Theorem: the number of Fqn rational points on XGS}, $\text{N}_{q}(\mathcal{\mathcal{X}_{\mathcal{G}_{\mathcal{S}}}})=\text{N}_{q^{2}}(\mathcal{\mathcal{X}_{\mathcal{G}_{\mathcal{S}}}})$. Further, if $p>3$ and $p\equiv -1\pmod{3}$, then $\text{N}_{q}(\mathcal{\mathcal{X}_{\mathcal{G}_{\mathcal{S}}}})=\text{N}_{q^{3}}(\mathcal{\mathcal{X}_{\mathcal{G}_{\mathcal{S}}}})$. Therefore, the result follows from Theorem \ref{Theorem: corollary of the main result of [34]} and Lemma \ref{Lemma: other properties on the L-polynomial}.
%From the proof of Theorem \ref{Theorem: the L-polynomial of XGS} in Section \ref{Section: Proof of Theorem 2} and a \textcolor{red}{straightforward} calculation, the following holds. From Theorem \ref{Theorem: the number of Fqn rational points on XGS}, Theorem \ref{Theorem: the L-polynomial of XGS} , and Lemma \ref{Lemma: other properties on the L-polynomial},
%\begin{itemize}
%%--------------------------------------------------------------------
%\item[\emph{1.}] 
%%--------------------------------------------------------------------
%\item[\emph{2.}]$L_{\mathcal{X}_{\mathcal{G}_{\mathcal{S}}}/\mathbb{F}_{q^{2}}}(T)=\bigg((qT+(-1)^{\frac{p+1}{2}})\cdot\Phi_{p}((-1)^{\frac{p-1}{2}}qT)\bigg)^{\frac{2g}{p}}$. 
%%--------------------------------------------------------------------
%\item[\emph{3.}] $ L_{\mathcal{X}_{\mathcal{G}_{S}} / \mathbb{F}_{q^2}}(1)=L_{\mathcal{X}_{\mathcal{G}_{S}} / \mathbb{F}_{q}}(1) \cdot \Big(p^{p-1}\mathcal{M}_{p}(-p^{t-1})^2   \cdot ((-1)^{\frac{p+1}{2}}+q) \Big)^{g/p}$. In particular, 
%\begin{eqnarray}
%\label{Eq2}
%L_{\mathcal{X}_{\mathcal{G}_{\mathcal{S}}}/\mathbb{F}_{q^{2}}}(1)>L_{\mathcal{X}_{\mathcal{G}_{\mathcal{S}}}/\mathbb{F}_{q}}(1).
%\end{eqnarray}
%\end{itemize}
\end{proof}
%-------------------------------------------------------------------------------------------------------------------------------------------------------------

%-------------------------------------------------------------------------------------------------------------------------------------------------------------
\begin{example}
To apply Corolary \ref{Corollary: a way to decide when the rational points generate the class group}, an explicit description of the well-known equality 
$$\dfrac{L_{\mathcal{X}_{\mathcal{G}_{\mathcal{S}}}/\mathbb{F}_{q^{2}}}(1)}{L_{\mathcal{X}_{\mathcal{G}_{\mathcal{S}}}/\mathbb{F}_{q}}(1)}=L_{\mathcal{X}_{\mathcal{G}_{\mathcal{S}}}/\mathbb{F}_{q}}(-1)$$ is desirable. Based on Theorem \ref{Theorem: the L-polynomial of XGS}, the following examples illustrate the cases $p=5,7$.\\

\begin{itemize}
%------------------------------------------------------------------------------
\item[$\pmb{p=5.}$] Here, 
\begin{equation*}
5^{2}\cdot\mathcal{M}_{5}(T)=25T^{4}+25T^{3}+15T^{2}+5T+1,
\end{equation*}
where $\mathcal{M}_{5}(T)$ is the minimal polynomial of $-\zeta_{5}/5^{1/2}$ over $\mathbb{Q}$. Therefore, 
\begin{equation*}
5^{2}\cdot\mathcal{M}_{5}(5^{t-1}T)=q^{2}T^{4}+qq_{0}T^{3}+3qT^{2}+q_{0}T+1,
\end{equation*}
which gives that
\begin{equation*}
L_{\mathcal{X}_{\mathcal{G}_{\mathcal{S}}}/\mathbb{F}_{q}}(T)=\bigg((qT^{2}-1)\cdot(q^{2}T^{4}+qq_{0}T^{3}+3qT^{2}+q_{0}T+1)^{2}\bigg)^{\frac{g}{5}}.
\end{equation*}
In particular,
$$
L_{\mathcal{X}_{\mathcal{G}_{S}} / \mathbb{F}_{q}}(1)=\left(q -1\right)^{\frac{g}{5}} \cdot\left(q^{2} +q q_{0} +3 q +q_{0} +1\right)^{\frac{2g}{5}}
$$
and
$$L_{\mathcal{X}_{\mathcal{G}_{\mathcal{S}}}/\mathbb{F}_{q}}(-1)=\left(q -1\right)^{\frac{g}{5}}
\left(q^{2} -q q_{0} +3 q -q_{0} +1\right)^{\frac{2g}{5}}.
$$
%------------------------------------------------------------------------------
\item[$\pmb{p=7.}$] Analogous to the previous case, 
\begin{equation*}
L_{\mathcal{X}_{\mathcal{G}_{\mathcal{S}}}/\mathbb{F}_{q}}(T)=\bigg((qT^{2}+1)\cdot(q^{3}T^{6}+q^{2}q_{0}T^{5}+3q^{2}T^{4}+qq_{0}T^{3}+3qT^{2}+q_{0}T+1)^{2}\bigg)^{\frac{g}{7}}.
\end{equation*}
In particular,
$$
L_{\mathcal{X}_{\mathcal{G}_{S}} / \mathbb{F}_{q}}(1)=\left(q +1\right)^{\frac{g}{7}} \cdot(q^{3}+q^{2}q_{0}+3q^{2}+qq_{0}+3q+q_{0}+1)^{\frac{2g}{7}}
$$
and
$$L_{\mathcal{X}_{\mathcal{G}_{\mathcal{S}}}/\mathbb{F}_{q}}(-1)=\left(q +1\right)^{\frac{g}{7}} \cdot(q^{3}-q^{2}q_{0}+3q^{2}-qq_{0}+3q-q_{0}+1)^{\frac{2g}{7}}.$$
\end{itemize}
\end{example}
%-------------------------------------------------------------------------------------------------------------------------------------------------------------

\bibliographystyle{amsplain}

\end{document}